\documentclass[12pt]{amsart}
\usepackage{amsmath,amscd}  

\numberwithin{equation}{section}

\newtheorem{theorem}[subsection]{Theorem}
\newtheorem{lemma}[subsection]{Lemma}
\newtheorem{proposition}[subsection]{Proposition}
\newtheorem{corollary}[subsection]{Corollary}

\newtheorem{remark}[subsection]{Remark}

\theoremstyle{definition}

\theoremstyle{remark}

\newtheorem{example}[subsection]{Example}

\newcommand{\F}{{\mathbb F}}
\newcommand{\C}{{\mathbb C}}
\newcommand{\Z}{{\mathbb Z}}
\newcommand{\M}{{M_m \otimes M_n}}
\newcommand{\MZ}{{M_m(\Z) \otimes M_n(\Z)}}
\newcommand{\V}{{V_m \otimes V_n}}
\newcommand{\field}{{\mathbb{F}}}
\newcommand{\kfield}{{\mathbb{F}}}

\newcommand{\GL}{{\rm GL}}
\newcommand{\SL}{{\rm SL}}
\newcommand{\Det}{{\rm Det}}
\newcommand{\KM}[1]{{{\mathcal K}M_{#1}}}
\newcommand{\KV}[1]{{{\mathcal K}V_{#1}}}
\newcommand{\Sym}{{\rm Sym}}
\newcommand{\tr}{{\rm Tr}}
\newcommand{\norm}{{\boldmath N}^{C_p}}
\DeclareMathOperator{\Span}{\rm span}
\DeclareMathOperator{\gr}{\rm gr}

\newcommand{\LT}{\mathop{\rm LT}}
\newcommand{\Deg}{\mathop{\deg_*}}
\newcommand{\LM}{\mathop{\rm LM}}

\renewcommand{\b}[1]{{\overline{#1}}}
\newcommand{\binomial}[2]{\genfrac{(}{)}{0pt}{}{#1}{#2}}

\newcommand\name[1]{{\sc #1\/}}

\newcommand{\FF}{{\mathcal F}}
\newcommand{\N}{\mathbb{N}}
\newcommand{\Q}{\mathbb{Q}}

\newcommand{\fp}{{\mathbb{F}}_p}
\newcommand{\Fp}{{\mathbb{F}}_p}

\newcommand{\Rsl}{{\text{Rep}_{\C \,\SL_2(\C)}}}
\newcommand{\Rz}{{\text{Rep}'_{\Q \,\Z}}}
\newcommand{\Rcp}{{\text{Rep}_{\Fp\, C_p}}}
\newcommand{\Ch}{\text{U}}  

\newcommand{\var}[1]{\widetilde{#1}}

\advance\vsize by -0.6cm
\title[Invariants of the modular cyclic group ]{Invariants for the modular cyclic group of prime order via classical invariant theory}

\author[Wehlau]{David L. Wehlau}
\address{Department of Mathematics and Computer Science \\ \hfil\break\indent
        Royal Military College \\ King\-ston, Ontario, Canada \\ K7K 5L0
             }
\email{wehlau@rmc.ca}



\begin{document}

\advance\vsize by -1cm

\begin{abstract}
Let $\field$ be any field of characteristic $p$.
It is well-known that there are exactly $p$ inequivalent indecomposable representations
$V_1,V_2,\dots,V_p$ of $C_p$ defined over $\F$.
Thus if $V$ is any finite dimensional $C_p$-representation there are non-negative integers
$0\leq n_1,n_2,\dots, n_k \leq p-1$ such that $V \cong \oplus_{i=1}^k V_{n_i+1}$.
It is also well-known there is a unique (up to equivalence) $d+1$ dimensional irreducible
complex representation of $\SL_2(\C)$ given by its action on the space $R_d$ of $d$ forms.
Here we prove a conjecture, made by R.~J.~Shank, which reduces the
computation of the ring of $C_p$-invariants $\field[ \oplus_{i=1}^k  V_{n_i+1}]^{C_p}$ to the computation
of the classical ring of invariants (or covariants) $\C[R_1 \oplus (\oplus_{i=1}^k  R_{n_i})]^{\SL_2(\C)}$.
This shows that the problem of computing modular $C_p$ invariants is equivalent to the problem of computing
classical $\SL_2(\C)$ invariants.

    This allows us to compute for the first time the ring of invariants for many representations of $C_p$.
  In particular, we easily obtain from this generators for the rings of vector invariants
  $\field[m\,V_2]^{C_p}$, $\field[m\,V_3]^{C_p}$ and $\field[m\,V_4]^{C_p}$for all $m \in \N$.  This is the first computation
  of the latter two families of rings of invariants.
\end{abstract}

\maketitle

\tableofcontents

\section{Introduction}\label{intro}
   Let $B$ be a domain and $G$ a finite group.  Consider a $BG$-module $V$ which is a free $B$-module of rank $n$.
      We write $B[V]$ to denote the symmetric algebra $\Sym_B^\bullet(V^*)$ on the dual $V^*$.
       If we fix a basis $\{x_1,x_2,\dots,x_n\}$ for $V^*$ we may identify $B[V]$ with the  polynomial ring  $B[x_1,x_2,\dots,x_n]$.
   The action of $G$ on $V$ induces an action of $G$ on $V^*$.
    Extending this action algebraically we get a natural action of $G$ on $B[V]$.  We write $B[V]^G$ to denote the subring of
    invariants:
    $$B[V]^G := \{f \in \kfield[V] \mid g\cdot f = f\ \forall g \in G\}\ .$$
    Emmy Noether(\cite{N1,N2}) proved that the ring $B[V]^G$ is always finitely generated when $B$ is a field (and $G$ is finite).

    We are concerned here with finding generators for the ring of invariants when $B=\kfield$ is a field of characteristic $p$ and $G=C_p$ is the
    cyclic group of order $p$.
  We want to describe generating sets for the $\field[V]^{C_p}$ not just for certain values of $p$ but
  rather for arbitrary primes $p$.

   The group $C_p$ has, up to equivalence, exactly $p$ indecomposable representations over $\field$.
 There is one indecomposable representation
 $V_n$ of dimension $n$ for every $n=1,2,\dots,p$.
   The representation $V_1$ is the trivial representation and
 $V_p$ is the regular representation.  If $V$ contains a copy of $V_1$ as a summand, say $V=V_1 \oplus V'$ then
 it is easy to see that $\field[V]^{C_p} = \field[V_1] \otimes \field[V']^{C_p}$.  For this reason it suffices to consider
 representations $V$ which do not contain $V_1$ as a summand.  Such a representation is called {\em reduced}.
 In order to simply the exposition we will assume that our representations of $C_p$ are reduced.

  In 1913, \name{L. Dickson}(\cite{Dickson})
computed the rings of invariants $\field[V_2]^{C_p}$ and $\field[V_3]^{C_p}$.
In 1990, \name{David Richman}(\cite{richman}) conjectured
a set of generators for \hbox{$\F[V_2 \oplus V_2\oplus \dots\oplus V_2]^{C_p}$}
(for any number of copies of $V_2$).  \name{Campbell} and
\name{Hughes}(\cite{campbell-hughes}) proved in 1997 that Richman's conjectured
set of generators were correct.

  In 1998, \name{Shank}(\cite{shank}) introduced a new method exploiting SAGBI bases
and found generating sets for the two rings of invariants $\F[V_4]^{C_p}$ and $\F[V_5]^{C_p}$.
In 2002, \name{Shank} and \name{Wehlau}(\cite{ndec}) extended Shank's method to find generators for
basis for $\F[V_2\oplus V_3]^{C_p}$.  Since then, Shanks's method has been used to find generators for
$\F[V_3\oplus V_3]^{C_p}$ (\cite{CFW}) and for $\F[V_2 \oplus V_2 \oplus V_3]^{C_p}$
(\cite{DLW}).  Limitations of this method
using SAGBI bases mean that  it seems infeasible to use the method to compute invariants for any further representation
of $C_p$ except probably $\F[V_2 \oplus V_4]^{C_p}$.    See \cite{jim's contribution} and \cite{CFW} for discussions of some of these limitations.
Thus $\field[V]^{C_p}$ is known (for general $p$) only for the infinite family $V=m\,V_2= \oplus^m V_2$ and for seven other small representations.

  There is a deep connection between the invariants of $C_p$ in characteristic $p$ and the classical invariants of $\SL_2(\C)$.
   This connection was pointed out and studied extensively by \name{Gert Almkvist}.
  See \cite{Alm1, reciprocity theorems, Alm3, Alm4, some formulas, Alm6, A-F}.

  Here  we prove a conjecture of \name{R. J. Shank} which reduces the computation of generators
for $\field[V]^{C_p}$ to the classical problem of computing $\C[W]^{\SL_2(\C)}$.  Here $W$ is a representation of $\SL_2(\C)$ which is
easily obtained from $V$ and with $\dim_{\C} W = \dim_{\field} V + 2$.   The invariant ring $\C[W]^{\SL_2(\C)}$ is called a ring of covariants (definition below).
Since generators for $\field[V]^{C_p}$
yield generators for $\C[W]^{\SL_2(\C)}$, our proof of this conjecture demonstrates the equivalence of these two problems.

  After giving our proof of the conjecture we use the computation of $\C[W]^{\SL_2(\C)}$ by classical invariant
theorists (and others) for a number of rings of covariants to give generators for the corresponding rings $\field[V]^{C_p}$.
This greatly extends the above list of representations of $C_p$ whose rings of invariants are known.

\section{Preliminaries}\label{prelim}
  We consider the $n \times n$ matrix with all eigenvalues equal to 1 and consisting of a single Jordan block:
$$ \sigma_n(B) :=
   \begin{pmatrix}
   1 & 0 & 0 & \hdots & 0 & 0\\
   1 & 1 & 0 & \hdots & 0  & 0\\
   0 & 1 & 1 & \hdots & 0  & 0\\
   \vdots & \vdots & \vdots & \ddots & \vdots & \vdots\\
   0 & 0 & 0 & \hdots & 1 & 0\\
   0& 0 & 0 & \hdots  & 1 & 1\\
 \end{pmatrix}_{n \times n}
$$
where the entries of this matrix are elements of the ring $B$.  Thus $\sigma_n(B) \in \GL_n(B)$.

The matrix $\sigma_n(B)$ generates a cyclic subgroup of $\GL_n(B)$.
If the characteristic of $B$ is 0 then $\sigma_n(B)$ has infinite order and so generates a group
isomorphic to $\Z$.
 It is not too hard to see that if the characteristic of $B$
is $p > 0$ then the order of $\sigma_n(B)$ is $p^r$ where $r$ is the least non-negative integer
such that $p^r \geq n$.

\subsection{Certain $\Z$-modules}
 We write $M_{n}$ to denote the $n$ dimensional $\Q$ vector space which is a $\Z$-module 
 where $1 \in \Z$ is represented by the matrix $\sigma_n(\Q)$.   It is easy to see that 
 this $\Z$-module satisfies $M_n^* \cong M_n$. 

  We fix a basis of $M_n$ with respect to which the matrix takes its given form.
 We write $M_{n}(\Z)$ to denote the rank $n$ lattice in $M_n$ generated by integer linear combinations of this fixed basis.
Thus $M_n(\Z) \otimes_{\Z} \Q = M_n$.
The action of $\Z$ on $M_n$ restricts to an action of $\Z$ on $M_{n}(\Z)$.
We write $\sigma$ to denote $\sigma_n(\Z)$ and $\Delta$ to denote $\sigma - 1$, an element of the group algebra.

    Note that the one dimensional $\Q$ vector space $M_n^{\Z}$ is the kernel of the map $\Delta : M_n \to M_n$.
  Given $W\cong \oplus_{i=1}^s M_{n_i}$ and $\omega \in W^{\Z}$ we say that
the {\em length} of $\omega$ is $r$ and write $\ell(\omega)=r$
to indicate that $r$ is maximal such that $\omega \in \Delta^{r-1}(W)$.

  \subsection{$C_p$-modules in characteristic $p$}
 The book \cite{CW} includes a description of the representation theory of $C_p$ over a field of characteristic $p$.
 We use $\sigma$ to denote a generator of the group $C_p$.
 We also consider $\Delta := \sigma - 1$, an element of the group algebra of $C_p$.
Whether $\sigma$ is a generator of $\Z$ or $C_p$ will be clear from the context.  Similarly the meaning of $\Delta$ will be clear from the context.

  Up to isomorphism, there is one indecomposable $C_p$ module of dimension $n$ for each $1 \le n \le p$.
  We denote this module by $V_n$.  Note that $V_n \cong V_n^*$.    
  Also $V_n$ is projective if and only if it is free if and only if $n=p$.

  Since the group $C_p$ is generated by a single element all of whose eigenvalues are 1, it follows that every $C_p$-module $V$
  is in fact defined over the prime field $\field_p \subseteq \field$.
  Thus if we let $V(\field_p)$ denote the $\field_p$-points of $V$ we have $V = V(\field_p) \otimes_{\field_p} \field$.
  Since $\field[V]^{C_p}$ is the kernel of the linear operator $\Delta:\field[V] \to \field[V]$ we see that
  $\field[V]^{C_p} = (\field_p[V(\field_p)]^{C_p}) \otimes_{\field_p} \field$.
  Therefore it suffices to work over the prime field $\field_p$.  We do this from now on.

  Usually $V_n$ is defined as the $n$-dimensional $\field_p$-module with the action of $\sigma$ given by the matrix
$\sigma_n(\field_p)$.  We will use an equivalent description that is somewhat less common.
We will realize $V_n$ as the quotient ring
$\F_p[t]/(t^n)$ equipped with a $C_p$ action by declaring that $\sigma$ acts via multiplication by $1+t$.
Of course with respect to the basis of monomials in $t$, the matrix representation of multiplication by $1+t$ is $\sigma_n(\field_p)$.
  We use this description of $V_n$ since it has an obvious grading given by polynomial degree.
More precisely, given an element of $\F_p[t]/(t^n)$ we use its unique representation as a linear combination of
$\{1,t,t^2,\dots,t^{n-1}\}$ in order to give it a well-defined degree.
We will use this polynomial degree to realize a filtration of $V_n$.
We define $\FF_r(V_n) := \{h \in \F_p[t]/(t^n) \mid \deg(h) \geq r\}$ for $0 \leq r \leq n$.
Then $\{0\}=\FF_{n}(V_n) \subset \FF_{n-1}(V_n) \subset \dots \subset \FF_0(V_n) = V_n$.

Any element of $V_n \setminus \Delta(V_n)$ generates the cyclic $C_p$-module $V_n$.
Let $\alpha$ denote such a generator and define $\omega := \Delta^{n-1}(\alpha)$.
Then  $V_n^{C_p}$, the socle of $V_n$, is spanned by $\omega$.
Given a $C_p$-module $W$ and $\omega \in W^{C_p}$ we
define  $\ell(\omega)$ to be the maximum integer $r$ such that
say that $\omega \in \Delta^{r-1}(W)$.  This integer $\ell(\omega)$
is called the {\em length} of $\omega$.

\subsection{Reduction modulo $p$}
  Let $p$ be a prime integer.
     Since $M_n(\Z)$ is a free $\Z$-module of rank $n$, reduction modulo $p$
  yields a surjective map $\rho : M_{n}(\Z) \to V := \field_p^{n}$.  The action of $\Z$ on
   $M_{n}(\Z)$ (generated by the action of $\sigma_n(\Z)$)
    induces an action on $V$ (generated by the action of $\sigma_n(\field_p))$.     Suppose now that $1 < n \leq p$   so that $\sigma_n(\field_p)$ has order $p$ and so gives an action
   of $C_p$ on $V$.   This action of $C_p$ on $V$ is indecomposable and therefore
   $V \cong V_n$ as a $C_p$-module.  
     Thus reduction modulo $p$ yields a surjective map
     $\rho : M_n(\Z) \to V_n$.  Both $M_n$ and $V_n$ are self dual and thus 
     reduction modulo $p$ is also surjective on the duals: $\rho : M_n^*(\Z) \to V_n^*$. 
     This map of duals in turn induces a surjective map of coordinate rings
  $$\rho : \Sym_{\Z}^\bullet(M_n^*(\Z)) = \Z[M_{n}(\Z)] \to \F_p[V_{n}] = \Sym_{\field}^\bullet(V^*) \ .$$

  More generally, reduction modulo $p$ gives a surjection
   $$\rho :   \Z[\oplus_{i=1}^k M_{n_i}(\Z)] \to \F_p[\oplus_{i=1}^k V_{n_i}]\ .$$
Since $\rho\circ\sigma_n(\Q) = \sigma_n(\field_p) \circ \rho$ we see that
 $$ \rho   (\Z[\oplus_{i=1}^k M_{n_i}(\Z)]^{\Z}) \subseteq \F_p[\oplus_{i=1}^k V_{n_i}]^{C_p}\ .$$
   Since $C_p$ is not linearly reductive (over $\field_p$)
 this may in fact be a proper inclusion.
We call the elements of $\rho(\Z[\oplus_{i=1}^k M_{n_i}(\Z)]^{\Z})$ {\em integral invariants}.
We caution the reader that Shank(\cite{shank,jim's contribution}) calls elements of
$\Z[\oplus_{i=1}^k M_{n_i}(\Z)]^{\Z}$ integral invariants and
elements of $\rho(\Z[\oplus_{i=1}^k M_{n_i}(\Z)]^{\Z})$ rational invariants.

\subsection{Invariants of $C_p$}

  Let $V$ be a $C_p$ representation.  For each $f \in \field_p[V]$ we define an invariant called the {\em transfer} or {\em trace} of $f$
  and denoted by $\tr(f)$ by
  $$\tr(f) := \sum_{\tau \in C_p} \tau f\ .$$
  Similarly we define the {\em norm of $f$}, denoted $\norm(f)$ by
  $$\norm(f) = \prod_{\tau \in C_p} \tau f\ .$$

  Consider a representation $V = V_{n_1} \oplus V_{n_2} \oplus \dots \oplus V_{n_r}$ of $C_p$.
 For each summand $V_{n_i}$ choose a generator $z_i$ of the dual cyclic $C_p$-module $V^*_{n_i}$,
 i.e., choose $z_i \in V^*_{n_i} \setminus \Delta(V^*_{n_i})$.    Define $N_i := \norm(z_i)$ for $i=1,2,\dots,r$.
  Later we will study $C_p$-invariants using a term order.  For a summary of term orders see \cite[Chapter 2]{CLO}.
We will always use a graded reverse lexicographic order with $z_i > \Delta(z_i) > \dots > \Delta^{n_i-1}(z_i)$ for
 all $i=1,2,\dots,r$.
    We denote the lead term of an element $f \in \field_p[V]$ by $\LT(f)$ and the lead monomial of
 $f$  by $\LM(f)$.  We follow the convention that a monomial is a product of variables.

 \section{The Conjecture}
    Let  $V=V_{n_1} \oplus V_{n_2} \oplus \dots \oplus V_{n_r}$ be a $C_p$-module.
   We have seen three ways to construct $C_p$ invariants: norms, traces and integral invariants.
   \name{R.J.~Shank}(\cite[Conjecture~6.1]{shank}) conjectured that  $\field_p[V]^{C_p}$ is generated by the norms  $N_1,N_2,\dots,N_r$
   together with a finite set of integral invariants and a finite set of transfers.  Originally Shank stated his conjecture only for $V$
   indecomposable but he later asserted it for general $C_p$-modules (\cite[\S 3]{jim's contribution}).
   Our main result here is to prove this conjecture.  We then apply the result to obtain generating sets for a number of
   $C_p$-modules $V$.

\section{Classical Invariant Theory of $\SL_2(\C)$}\label{classical}

Here we consider representations of the classical group $\SL_2(\C)$.  There are many
good introductions to this topic.  For our purposes the book by \name{Procesi}(\cite{Pr}) is especially well suited
since it emphasizes an invariant theoretic approach.
The results of this section are well-known.

Let $R_1$ denote the defining representation of $\SL_2(\C)$ with basis $\{X,Y\}$.
Define $R_d := \Sym^d(R_1)$ to be the space of homogeneous forms of degree $d$ in $X$ and $Y$.
The action of $\SL_2(\C)$ on $R_1$ induces an action on $R_d$.  This
action\footnote{In fact, classically the formula used was
$\begin{pmatrix}a & b\\
c & d\end{pmatrix} \cdot f(X,Y) = f(aX+bY,cX+dY)\ .$  This yields a right action and since we prefer left actions
we use the other formula.  It is clear the two actions are equivalent and have the same ring of invariants.}
 is given by
$$\begin{pmatrix}a & b\\
c & d\end{pmatrix} \cdot f(X,Y) = f(aX+cY,bX+dY)\ .$$

 \name{Gordan}\cite{G} showed that the algebra $\C[W]^{\SL_2(\C)}$ is finitely generated for any finite dimensional representation $W$ of $\SL_2(\C)$.
  The algebra $(\Sym^\bullet(R_1) \otimes \C[W])^{\SL_2(\C)}$ is known as the ring of covariants of $W$.
    This ring was a central object of study in classical invariant theory.   Since the representations $R_1$ and $R_1^*$
    are equivalent, it follows that
     $(\Sym^\bullet(R_1) \otimes \C[W])^{\SL_2(\C)} \cong \C[R_1 \oplus W]^{\SL_2(\C)}$.
    We will also refer to this latter ring as the ring of covariants of $W$.
      Classical invariant theorists found generators for the rings of covariants
of a number of small representations $W$ of $\SL_2(\C)$.

We work with the basis $\{x,y\}$ of
$R_1^*$ which is dual to the basis $\{Y,X\}$ of $R_1$.
Then $\sigma(x)=y-x$ and $\sigma(y)=y$.
We also use $\{\binomial d i  a_i \mid i=0,1,\dots,n\}$ as a basis for $R_d^*$ where
$\{a_0,a_1,\dots,a_d\}$ is dual to $\{X^d,-X^{d-1}Y,\dots,X^{d-i}(-Y)^i,\dots,Y^d\}$.
We choose these bases in order that the homogeneous $d$-form
$$f = \sum_{i=0}^n \binomial d i  a_i x^{d-i} y^i \in (R_1 \oplus R_d)^*$$
is invariant under the action of $\SL_2(\C)$.
Putting $f = \sum_{i=0}^n \binomial d i  a_1 x^{d-i} y^i$ into the above formula for the action
we find that
$\sigma = \sigma_2(\C)$
acts on $R_d^*$ via $\sigma(a_r) = \sum_{j=0}^r \binomial r j a_j$ for $r=0,1,\dots,d$.
 From this it is easy to see that $\sigma$ acts irreducibly
on $R_d^*$.  It can be shown that $R_d$ and $R_d^*$ are equivalent as representations of $\SL_2(\C)$.
In fact, if $W$ is any irreducible representation of $\SL_2(\C)$ of dimension $d+1$
then $W$ is equivalent to $R_d$.
  Since $\sigma$ acts irreducibly on $R_d$, it follows that
the action of $\sigma$  %
on $R_d$ is given (with respect to a Jordan basis) by $\sigma_{d+1}(\C)$.

  Given two forms $g \in R_m = \Sym^m(R_1)$ and $h \in R_n = \Sym^n(R_1)$, their
  {\em $r^\text{th}$ transvectant} is defined by
  $$(g,h)^r := \frac{(m-r)!(n-r)!}{m!n!}\sum_{i=0}^r (-1)^i \binomial{r}{i} \frac{\partial^r g}{\partial X^{r-i}\partial Y^i}  \frac{\partial^r h}{\partial X^{i}\partial Y^{r-i}}$$
  for $r=0,1,\dots,\min\{m,n\}$.  It has degree (traditionally called order) $m+n-2r$ in $X,Y$, i.e.,  $(g,h)^r \in R_{m+n-2r}$.

The Clebsch-Gordan formula  (\cite[\S3.3]{Pr}) asserts that   $$R_m \otimes R_n \cong \bigoplus_{r=0}^{\min\{m,n\}} R_{m+n-2r}\ .$$
If $g \in R_m$ and $h \in R_n$ then the projection of $R_m \otimes R_n$ onto its summand $R_{m+n-2r}$ carries
$g \otimes h$ onto $(g,h)^r$.

  \begin{example}\label{eg1}
         The ring of covariants of $W=R_2\oplus R_3$ was computed by classical invariant theorists.
     In this example, we concentrate on
     $\C[R_1 \oplus R_2 \oplus R_3]_{(*,1,1)}^{\SL_2(\C)}$.
      In \cite[\S 140]{G-Y} it is shown that the ring $\C[R_1 \oplus R_2 \oplus R_3]^{\SL_2(\C)}$  is generated by 15 generators.
      Following the notation there, we use $\phi$ to denote an element of $R_2$ (the quadratic) and $f$ to denote
      an element of $R_3$ (the cubic).  Examining the multi-degrees
      of the 15 generators we find that 4 of them are relevant to understanding $\C[R_1 \oplus R_2 \oplus R_3]_{(*,1,1)}^{\SL_2(\C)}$.
      These are $(\phi,f)^1$ of degree $(3,1,1)$,  $(\phi,f)^2$ of degree $(1,1,1)$ and the two forms $\phi$ of degree $(2,1,0)$ and
      $f$ of degree $(3,0,1)$.    Thus $\C[R_1 \oplus R_2 \oplus R_3]_{(*,1,1)}^{\SL_2(\C)}$ is 3 dimensional with basis
      $\{(\phi,f)^1,(\phi,f)^2,\phi f\}$.
  \end{example}

\section{Roberts' Isomorphism}
  Given a covariant $g = g_0 Y^d + g_1 X Y^{d-1} + \dots + g_d X^d$, the coefficient $g_0$ of $Y^d$ is called the
  {\it source} of $g$ and is also known as a {\it semi-invariant}.

Let $W$ be any representation of $\SL_2(\C)$.  Roberts' isomorphism (see \cite{roberts}) is the isomorphism which
associates to a covariant its source:
$$\psi :\C[R_1\oplus W]^{\SL_2(\C)} \to \C[W]^H$$ given by $\psi(f(\cdot,\cdot)) = f(Y,\cdot)$ where $R_1$
has basis $
\{X,Y\}$
and $H$ is the subgroup
$$H := \SL_2(\C)_{Y} = \{\alpha \in \SL_2(\C) \mid \alpha\cdot Y = Y\}
= \{\begin{pmatrix}
1 & 0\\
z & 1
\end{pmatrix} \mid z \in \C\}$$ which fixes $Y$.
For a modern discussion and proof of Roberts' isomorphism see \cite{BK}  or \cite[ \S15.1.3 Theorem~1]{Pr}.

  Clearly $H$ contains a copy $K$ of the integers $\Z$ as a dense (in the Zariski topology) subgroup:
$K :=  \{\begin{pmatrix}
1 & 0\\
m & 1
\end{pmatrix} \mid m \in \Z\}$.  This is just the subgroup of $\GL_2(\C)$ generated by $\sigma_2(\C)$.
Since $K$ is dense in $H$, we have  $\C[W]^H = \C[W]^{K}$ and thus
$\C[R_1\oplus W]^{\SL_2(\C)} \cong \C[W]^K$.

  Since the action of $K$ on $W$ is defined over $\Z \subset \Q$ we have
$$W \cong W(\Q) \otimes_{\Q} \C \cong W(\Z) \otimes_{\Z} \C$$
 where $W(\Z)$ denotes the integer points of $W$ and
$W(\Q)\cong W(\Z) \otimes_{\Z} \Q$ denotes the $\Q$ points of $W$.

 Thus
 $$
 \C[R_1 \oplus W]^{\SL_2(\C)} \cong \C[W]^H = \C[W]^K\\
 $$

 As above, $\C[W]^K$ is the kernel of the linear operator $$\Delta: \C[W] \to \C[W]$$ and thus
 $$\C[W]^K \cong \Q[W(\Q)]^K \otimes_{\Q} \C \cong  (\Z[W(\Z)]^K \otimes_{\Z} \Q) \otimes_{\Q} \C \cong \Z[W(\Z)]^K \otimes_{\Z} \C\ .$$

Clearly $R_d(\Q)$ is isomorphic to the $\Z$-module  $M_{d+1}$ considered above.  Thus
we may identify $M_{d+1}$ with the $\Q$-points of $R_d$ and $M_{d+1}(\Z)$ with the $\Z$ points of $R_d$.
Writing $W \cong \oplus_{i=1}^k R_{d_i}$ we have
 \begin{align*}
 \C[R_1 \oplus(\oplus_{i=1}^k R_{d_i})]^{\SL_2(\C)} &\cong \Q[\oplus_{i=1}^k M_{d_i+1}]^\Z \otimes_{\Q} \C \\
  &\cong \Z[\oplus_{i=1}^k M_{d_i+1}(\Z)]^\Z \otimes_{\Z} \C\ .
 \end{align*}
 (Here we are writing $\Z$ for the group $K$.)
Furthermore $$\rho: \Z[\oplus_{i=1}^k M_{d_i+1}(\Z)]^\Z  \to \F_p[\oplus_{i=1}^k V_{d_i+1}]^{C_p}$$
  where the kernel of $\rho$ is the principal ideal generated by $p$.

\section{Periodicity}
  Let $V$ be a $C_p$-module and write $V$ as a direct sum of indecomposable $C_p$-modules:
 $V = V_{n_1} \oplus V_{n_2} \oplus \dots \oplus V_{n_r}$.  This decomposition induces a $\N^r$ grading on $\field_p[V]$
 which is preserved by the action of $C_p$.   As above we choose a generator $z_i \in V^*_{n_i}$ for each $i=1,2,\dots,r$
 and put $N_i := \norm(z_i)$.  We further define $\field_p[V]^\sharp$ to be the ideal of $\field_p[V]$ generated by $N_1,N_2,\dots,N_r$.

 The following theorem (see for example \cite[\S 2]{ndec}) is very useful.
 \begin{theorem}[Periodicity]
   The ideal $\field_p[V]^\sharp$ is a summand of the $C_p$-module $\field_p[V]$.  Denoting its complement by $\field_p[V]^\flat$
   we have the decomposition $\field_p[V] = \field_p[V]^\sharp \oplus \field_p[V]^\flat$ as a $C_p$-modules.  Taking the multi-grading
   into account we have
   $$\field_p[V]_{(d_1,d_2,\dots,d_r)} = \field_p[V]_{(d_1,d_2,\dots,d_r)}^\sharp \oplus \field_p[V]_{(d_1,d_2,\dots,d_r)}^\flat\ .$$
   Moreover if there exists $i$ such that $d_i \geq p-n_i+1$ then $\field_p[V]_{(d_1,d_2,\dots,d_r)}^\flat$ is a free $C_p$-module.
 \end{theorem}

 \begin{remark}
   More can be said.  In fact
   $$\field_p[V]_{(d_1,d_2,\dots,d_{i-1},d_i+p,d_{i+1},\dots,d_r)} \cong \field_p[V]_{(d_1,d_2,\dots,d_r)} \oplus k\, V_p$$
   for some positive integer $k$.  This explains why the the previous theorem is known by the name {\em periodicity}.
 \end{remark}

  The decomposition given by the periodicity theorem obviously yields a vector space decomposition of the
  multi-graded ring of invariants:
  $$\field_p[V]^{C_p}_{(d_1,d_2,\dots,d_r)} = (\field_p[V]^{C_p}_{(d_1,d_2,\dots,d_r)})^\sharp \oplus (\field_p[V]^{C_p}_{(d_1,d_2,\dots,d_r)})^\flat\ .$$
  Here $$(\field_p[V]^{C_p}_{(d_1,d_2,\dots,d_r)})^\sharp = (\field_p[V]_{(d_1,d_2,\dots,d_r)}^\sharp)\cap\field_p[V]^{C_p}$$ is the ideal
  of $\field_p[V]^{C_p}$ generated by  $N_1,N_2,\dots,N_r$ and $$(\field_p[V]^{C_p}_{(d_1,d_2,\dots,d_r)})^\flat = (\field_p[V]_{(d_1,d_2,\dots,d_r)}^\flat)\cap\field_p[V]^{C_p}\ .$$

\section{Outline of the proof}
We are now in position to outline the main steps of our proof.
     We want to show that the co-kernel of the reduction mod $p$ map $\rho:\Z[\oplus_{i=1}^r M_{n_i}(\Z)]^\Z \to \field_p[\oplus_{i=1}^r V_{n_i}]^{C_p}$ is spanned
     by products of transfers and the norms $N_1,N_2,\dots,N_r$.     We consider a fixed multi-degree $(d_1,d_2,\dots,d_n)$.
     Using the Periodicity Theorem we may reduce to the case where each $d_i < p$.  Then we may exploit the fact that for such values of $d_i$
     the homogeneous component $\field_p[\oplus_{i=1}^r V_{n_i}]_{(d_1,d_2,\dots,d_n)}$ is a summand of $\otimes_{i=1}^r \otimes^{d_i} V_{n_i}$.
     Thus we may consider the reduction mod $p$ map $\rho:  \otimes_{i=1}^r \otimes^{d_i} M_{n_i}(\Z) \to \otimes_{i=1}^r \otimes^{d_i} V_{n_i}$.
     For this map we will show that for any summand $V_k$ of $\otimes_{i=1}^r \otimes^{d_i} V_{n_i}$ with $k <p$, there exists a corresponding summand
     $M_k$ of  $\otimes_{i=1}^r \otimes^{d_i} M_{n_i}$ with $\rho(M_k(\Z)) = V_k$.  In particular $V_k^{C_p}$ lies in the image of $\rho$.
     By induction we reduce to $\rho : M_m(\Z) \otimes M_n(\Z) \to V_m \otimes V_n$ where $m,n \leq p$.   By carefully examining explicit decompositions of
     $M_m \otimes M_n$ and $V_m \otimes V_n$ we are able to show that any summand $V_k$ of $V_m \otimes V_n$ is contained $\rho(M_m(\Z) \otimes M_n(\Z))$.

   The following example is instructive as regards both dependence on the prime $p$ and and our solution to the last step in the above outline of the proof.

\begin{example}\label{eg2}
     We consider the $\Z$-module $M_3 \otimes M_4$.  We realize $M_3$ as $\Q[s]/(s^3)$ and $M_4$ as $\Q[t]/(t^4)$
     with $\Delta(s^i) = s^{i+1}$ and $\Delta(t^j) = t^{j+1}$.
          By (\ref{C-G for M}) we have $M_3 \otimes M_4 \cong M_2 \oplus M_4 \oplus M_6$.  
          We can also see this decomposition explicitly as follows.
    Let $\alpha_0 := 1$, $\alpha_1 := 3s-2t$ and $\alpha_2 := 3s^2-2st + t^2+ 2t^3$.
    Then
    \begin{align*}
    \alpha_0 &:= 1\\
      \Delta(\alpha_0) &=  s +  t + s t\\
      \Delta^2(\alpha_0) &=  s^2 + 2s t + t^2 + 2 s^2t + 2st^2 + s^2t^2\\
      \Delta^3(\alpha_0) &= 3s^2t + 3s t^2 + t^3 + 6 s^2t^2 + 3st^3 + 3s^2t^3\\
      \Delta^4(\alpha_0) &= 6s^2t^2 + 4s t^3 + 12s^2t^3\\
      \Delta^5(\alpha_0) &= 10s^2 t^3\\
      \Delta^6(\alpha_0) &=  0\\
   \alpha_1 &:= 3s-2t\\
      \Delta(\alpha_1) &= 3s^2 +st -2t^2 + 3s^2 t- 2st^2\\
      \Delta^2(\alpha_1) &= 4s^2t - s t^2 + 2t^3 +2 s^2t^2 -4st^3 + 2 s^2t^3\\
      \Delta^3(\alpha_1) &= 3s^2t^2 - 3s t^3 - 3s^2t^3\\
   \Delta^4(\alpha_1) &= 0\\
   \alpha_2 &:= 3s^2-2st + t^2+2t^3\\
      \Delta(\alpha_2) &= s^2t - st^2 + t^3- 2s^2t^2 + 3st^3\\
      \Delta^2(\alpha_2) &=0
      \end{align*}

Thus $\Span_{\Q}\{\alpha_2,\Delta(\alpha_2)\} \cong M_2$,
$\Span_{\Q}\{\Delta^j(\alpha_1) \mid 0 \leq i \leq 3\} \cong M_4$ and
$\Span_{\Q}\{\Delta^j(\alpha_0) \mid 0 \leq i \leq 5\} \cong M_6$.
Hence we have an explicit decomposition: $M_3 \otimes M_4 \cong M_2 \oplus M_4 \oplus M_6$.

We put $\omega_0 := s^2 t^3$, $\omega_1 := s^2t^2 - s t^3 - s^2t^3$
and  $\omega_2 := 
                              s^2t - st^2 + t^3- 2s^2t^2 + 3st^3$.
Thus $\omega_0 := \Delta^5(\alpha_0/10)$, $\omega_1 := \Delta^3(\alpha_1/3)$
and  $\omega_2 := \Delta^1(\alpha_2)$.
Therefore $\ell(\omega_0)=6$, $\ell(\omega_1)=4$ and $\ell(\omega_2)=2$.

Take $p \geq 5$.
Reduction modulo $p$ gives the map $\rho : M_3(\Z) \otimes M_4(\Z) \to V_3 \otimes V_4$.
From Proposition~\ref{decomp} we have
   $$V_3 \otimes V_4 \cong
   \begin{cases}
       V_2 \oplus V_4 \oplus V_6,& \text{if } p \geq 7;\\
       V_2 \oplus 2\,V_5,& \text{if } p=5.
    \end{cases}
    $$
Again we may see this decomposition explicitly by considering the action of $C_p$ on $V_3 \otimes V_4$ as follows.

Put $\b{\omega}_i := \rho(\omega_i)$ and $\b{\alpha}_i := \rho(\alpha_i)$ for $i=0,1,2$.

First suppose that $p \geq 7$.
Take $\mu_0,\mu_1\in \Z$ with
 $10\mu_0 \equiv 1 \pmod p$ and $3\mu_1 \equiv 1 \pmod p$.
Then from the above computations we have
$\Span_{\field_p}\{\b{\alpha}_2,\Delta(\b{\alpha}_2)\} \cong V_2$,
$\Span_{\field_p}\{\Delta^j(\mu_1\b{\alpha}_1) \mid 0 \leq i \leq 3\} \cong V_4$ and
$\Span_{\field_p}\{\Delta^j(\mu_0 \b{\alpha}_0) \mid 0 \leq i \leq 5\} \cong V_6$.
In particular,
 $\Delta^{5}(\mu_0 \b{\alpha}_0)=\b{\omega}_0$, $\Delta^{3}(\mu_1\b{\alpha}_1)=\b{\omega}_i$ and
 $\Delta(\b{\alpha}_2)=\b{\omega}_2$ and therefore
 $\ell(\b{\omega}_2)=2=\ell(\omega_2)$, $\ell(\b{\omega}_1)=4=\ell(\omega_1)$ and
 $\ell(\b{\omega}_0)=6=\ell(\omega_0)$.

Now we consider the case $p=5$.  Then $\Delta^5(\b{\alpha}_0)=0$ and from this we can show that
$\b{\omega}_0 \notin \Delta^5(V_3 \otimes V_4)$.  Hence $\ell(\b{\omega}_0) \leq 5$.  Here we may define
$\b{\beta}_0 := \rho(s)$,
 $\b{\beta}_1 :=  3\b{\alpha}_0$ and
$\b{\beta}_2 := \rho(\alpha)_2)$.
Then
$\Span_{\field_5}\{\b{\beta}_2,\Delta(\b{\beta}_2)\} \cong V_2$,
$\Span_{\field_5}\{\Delta^j(\mu_1\b{\beta}_1) \mid 0 \leq i \leq 4\} \cong V_5$ and
$\Span_{\field_5}\{\Delta^j(\mu_0 \b{\beta}_0) \mid 0 \leq i \leq 4\} \cong V_5$.
Then $\Delta^4(\b{\beta_0}) = \b{\omega}_0$,  $\Delta^4(\b{\beta}_1) = \b{\omega}_1$ and
 $\Delta(\b{\beta}_2) = \b{\omega}_2$.  Thus $\ell(\b{\omega}_0)=\ell(\b{\omega}_1)=5$ and $\ell(\b{\omega}_2)=2$.

 Comparing this with Example~\ref{eg1} we see that (up to choice of bases)
       $\psi( (\phi,f)^1 ) =  \omega_1$,  $\psi( (\phi,f)^2 ) =  \omega_2$ and  $\psi( \phi f ) =  \omega_0$.
\end{example}

\section{Representation Rings}
\subsection{Complex representations of $\SL_2(\C)$}
Let $\Rsl$ denote the representation ring of complex representations of $\SL_2(\C)$.
Then $$\Rsl \cong \Z[\var{R}_1] \cong \oplus_{d=0}^{\infty} \Z \var{R}_d\ .$$
Here $\var{R}_d$ is a formal variable corresponding to the representation
$R_d$ for all $d \geq 1$ and $\var{R}_0 = 1 \in \Rsl$ corresponds to the 1 dimensional
trivial representation.
The multiplication in $\Rsl$ is given by the Clebsch-Gordan rule: (see \cite[\S3.3]{Pr})
$$\var{R}_m \cdot \var{R}_n = \sum_{k=0}^{\min\{m,n\}} \var{R}_{|n-m|+2k}.$$
This formula can be used to inductively derive a formula expressing
$\var{R}_d$ as a polynomial in $\Z[\var{R}_1]$.
Almkvist  \cite[Theorem~1.4(a)]{some formulas}  showed that in fact
$\var{R}_d = \Ch_{d+1}(\var{R}_1/2)$  where
$\Ch_n(x)$ is the $n^\text{th}$ Chebyshev polynomial of the second kind.

\subsection{Certain rational representations of $\Z$}
Let $\Rz$ denote the subring of the representation ring of $\Z$ given by
$$\Rz := \Z[\var{M}_2] \cong \oplus_{d=1}^{\infty} \Z \var{M}_d\ .$$
Here $\var{M}_d$ is a formal variable corresponding to the representation
$M_d$ for all $d \geq 2$ and $\var{M}_1 = 1 \in \Rz$ corresponds to the 1 dimensional
trivial representation.
  The multiplication in
$\Rz$ is given by a Clebsch-Gordan type formula:
\begin{equation}  \label{C-G for M}
\var{M}_m \cdot \var{M}_n =
     \displaystyle \bigoplus_{k=1}^{\min\{m,n\}} \var{M}_{|n-m|+2k-1}\ .
\end{equation}

  This result is an immediate consequence of the Clebsch-Gordan formula for $R_{m-1} \otimes R_{n-1}$,
 after restricting from $\SL_2(\C)$ to the subgroup $K \cong \Z$ and using Roberts' isomorphism as above.
Alternatively a proof of this result follows from the Jordan Form of the Kronecker (or tensor) product of
two matrices in Jordan Form.  Such a decomposition was given independently by \name{Aiken}(\cite{A}) and
\name{Roth}(\cite{roth}) in 1934.  The proof by \name{Aiken} contains an error (the same error occurs in the treatment by
\name{Littlewood}(\cite{L}) of this problem).  The proof by \name{Roth} has been criticized for not providing
sufficient details for the so-called ``hard case'' when both matrices are invertible.  This is precisely the case
which may be settled by exploiting Roberts' isomorphism.
\name{Marcus} and \name{Robinson}(\cite{MR}) gave a complete proof extending the ideas of \name{Roth}.
In the words of \name{Bruldi} (\cite{Br}), ``the difficult case $\ldots$ constitutes the most substantial part of [Marcus and Robinson's]
proof''.  \name{Brualdi}(\cite{Br}) gave a proof based on \name{Aiken}'s method.  \name{Brualdi}'s article also includes a good discussion of the history
of this problem.     The approach via Roberts' isomorphism appears to have been overlooked by people studying this problem.

Again this formula can be used to inductively derive a formula expressing
$\var{M}_d$ as a polynomial in $\Z[\var{M}_2]$.  Once again the answer is given by
a Chebyshev polynomial of the second kind:
$\var{M}_d  = \Ch_d(\var{M}_2/2)$.

\subsection{Characteristic $p$ representations of $C_p$}
Let $\Rcp$ denote the representation ring of $C_p$ over the field $\Fp$.

The multiplication here is determined by
$$\var{V}_2 \otimes \var{V}_n \cong
   \begin{cases}
            \var{V}_2, & \text{ if } n=1;\\
            \var{V}_{n-1}\oplus \var{V}_{n+1}, & \text{ if } 2\leq n \leq p-1;\\
            2\,\var{V}_p, & \text{ if } n=p.
   \end{cases}
$$
Here $\var{V}_d$ is a formal variable corresponding to the representation
$V_d$ for all $2 \leq d \leq p$ and $\var{V}_1 = 1 \in \Rcp$ corresponds to the 1 dimensional
trivial representation.
For an especially simple proof of this formula see the proof of \cite[Lemma~2.2]{HK}.

From this it follows that $\var{V}_d  = \Ch_d(\var{V}_2/2)$ for $d \leq p$, a fact
also shown by Almkvist \cite[Theorem~5.10(b)]{reciprocity theorems}.
It is convenient to define $\var{V}_d  := \Ch_d(\var{V}_2/2) \in \Rcp$ for $d >p$.

The above formula implies  that
$$\Rcp \cong \Z[\var{V}_2] \cong \Z[T]/q(T) \cong \oplus_{d=1}^{p} \Z \var{V}_d$$
where $q$ is a certain polynomial of degree $p$.  For details see Almkvist's paper \cite{some formulas}.
 The polynomial $q$ is determined by
the fact that $\var{V}_{p+1} - 2 \var{V}_p + \var{V}_{p-1}=0$.
Thus $q(T) =  \Ch_{p+1}(T/2) - 2 \Ch_{p}(T/2) +  \Ch_{p-1}(T/2)$.

 Reduction modulo $p$ carries the lattice $M_d(\Z)$ to the
  representation $V_d$ of $C_{p^r}$ where $p^{r-1} < d \leq p^r$.
  In particular, reduction modulo $p$ carries $M_d(\Z)$ to the representation $V_d$
  of $C_p$ for all $d \leq p$.
     Thus the map $\rho$, defined above, induces a map
  $\phi : \Rz \to \Rcp$ given by $\phi(\var{M}_2)=\var{V}_2$.  Also
  $\phi(\var{M}_d) = \var{V}_d$ for all $d=1,2,\dots,p$.
  In fact $\phi(\var{M}_d) = \var{V}_d$ for all $d \geq 1$ since
  $\var{M}_d = \Ch_d(\var{M}_2/2)$ for all $d \geq 1$.

With this convention the multiplication rule may be expressed in a form similar to the Clebsch-Gordan formula:
\begin{equation}\label{C-G for V}
   \var{V}_m \cdot \var{V}_n = \sum_{k=1}^{\min\{m,n\}} \var{V}_{|n-m|+2k-1}.
\end{equation}
 It is clear that the map $\phi$ is a surjection whose kernel is the principal ideal generated by
   $q(\var{M}_2)$.

We wish to derive a more enlightening and explicit formula for the product $\var{V}_m \cdot \var{V}_n$
for the cases corresponding to actual (indecomposable) representations, i.e., when $1 \leq m,n \leq p$.
More precisely, we want to express such a product in terms of  the elements $\var{V}_d$ with $d \leq p$.

We prove the following.
\begin{proposition}\label{decomp}
  Let $1 \leq m \leq n \leq p$.  Then
$$\var{V}_m \cdot \var{V}_n =
   \begin{cases}
         \quad\,\,\, \displaystyle \sum_{i=1}^m \var{V}_{m+n-2i+1}
                 =\displaystyle \sum_{s=1}^m \var{V}_{n-m+2s-1}, & \hskip -1.5em\text{if }  m+n \leq p+1;\\
        \displaystyle\sum_{i=m+n-p+1}^{m} \var{V}_{m+n-2i+1} + ( m+n-p)\var{V}_p\\
            \qquad =     \displaystyle\sum_{s=1}^{p-n} \var{V}_{n-m+2s-1} + ( m+n-p)\var{V}_p, & \text{ if }  m+n \geq p.
                \end{cases}
$$
\end{proposition}
   \begin{proof}
     If $m+n \leq p+1$ then the result follows from the Clebsch-Gordan type formula (\ref{C-G for V}) above.

     Thus we suppose that $m+n \geq p+2$.  For this case, the proof is by induction on $m$.
     If $m=2$ then we must have $n=p$ since $m+n \geq p+2$.  Hence $V_n=V_p$ is projective and therefore so is
     $V_m \otimes V_n$.
     This implies $V_m \otimes V_p \cong m\,V_p$.  Thus the result is true for $m=2$.

     Suppose then that the result holds for $m=2,3,\dots,r-1$ and we will prove it true for $m=r$.
     Again, using projectivity, the result is clear if $m=p$ so we may suppose that $m=r \leq p-1$.
     Consider $\var{V}_2 \times \var{V}_{r-1} \times \var{V}_n$.
     \begin{align*}
            (\var{V}_2  \times \var{V}_{r-1}) \times \var{V}_n  & = (\var{V}_r + \var{V}_{r-2})\times \var{V}_n\\
           & = (\var{V}_r \times \var{V}_n)  + (\var{V}_{r-2} \times \var{V}_n) \\
            & = (\var{V}_r \times \var{V}_n) + (\sum_{s=1}^{p-n} \var{V}_{n-r+2s+1}) + ( r+n-p-2)\,\var{V}_p\\
               & \qquad\qquad             \text{ since }r+n-2 \geq p.
     \end{align*}
     Conversely
     \begin{align*}
          \var{V}_2  \times (\var{V}_{r-1} &\times \var{V}_n)   = \var{V}_2 \times (\sum_{s=1}^{p-n} \var{V}_{n-r+2s} + (r+n-p-1)\,\var{V}_p)\\
          &\qquad\text{ since } r-1\geq p-n+1 > p-n\\
           & =   \sum_{s=1}^{p-n} \var{V}_{n-r+2s+1} + \sum_{s=1}^{p-n} \var{V}_{n-r+2s-1} + 2(r+n-p-1)\,\var{V}_p\ .
     \end{align*}
     Therefore $\var{V}_r \times \var{V}_n =   \sum_{s=1}^{p-n} \var{V}_{n-r+2s-1}  + (r+n-p)\, \var{V}_p$.
   \end{proof}

   \begin{remark}
     Note that if $m \leq n \leq p$ then $\dim (V_m \otimes V_n)^{C_p} = m$ for both of the cases
     $m + n \leq p$ and $m + n > p$.
   \end{remark}

\section{Explicit Decompositions}
The formulae in the previous section describe the decomposition of tensor products of representations abstractly.
We will require more explicit decompositions including not just a list of the representations occurring in a product
but also some information about how these subrepresentations lie in the tensor product.  We will use the formulae from the
 representation rings to help in determining this extra information.

\subsection{Decomposing $M_m \otimes M_n$}
We begin by considering the product $M_m \otimes M_n$.
 Suppose $m \leq n$.  From the representation ring formula we know that
\begin{equation}\label{explicit CG for M}
  M_m \otimes M_n \cong \displaystyle \bigoplus_{s=1}^m M_{n-m+2s-1}\ .
\end{equation}

  For our purposes we need a explicit description of the submodules occurring in this decomposition.
We write
$$M_m \otimes M_n \cong \frac{\Q[s]}{(s^m)} \otimes \frac{\Q[t]}{(t^n)}
                                     \cong \frac{\Q[s,t]}{(s^m, t^n)}$$
which we identify with $\Span_{\Q}\{s^it^j \mid 0 \leq i < m, 0 \leq j < n\}$.
$\Z$ acts on $M_m \otimes M_n$ via $\sigma=(1+s)(1+t) = 1 + s + t + st$
and $\Delta = \sigma-1 = s + t + st$.

 We filter $M_m \otimes M_n$ by total degree writing $\FF_r(M_m \otimes M_n) := \{ h \in  \Q[s,t]/(s^m, t^n) \mid \deg(h) \geq r\}$.
For $h \in M_m \otimes M_n$, we write $\gr(h)=h_d$ where $h=h_d + h_{d+1} + \dots + h_{m+n-2}$
with $h_d \neq 0$ and $h_r \in (M_m \otimes M_n)_r$.
For $0 \neq h \in M_m \otimes M_n$, we define $\Deg(h) = \deg(\gr(h))$.
 We consider the Hilbert function of $M_m \otimes M_n$ defined by
 $$
   H(M_m \otimes M_n,j) := \dim (M_m \otimes M_n)_j
 $$

   The Hilbert series of $M_r$ is the polynomial $1 + \lambda + \lambda^2 + \dots + \lambda^{r-1} = \frac{1-\lambda^r}{1-\lambda}$.
   Thus the Hilbert series of $M_m \otimes M_n$ is given by  $\frac{1-\lambda^m}{1-\lambda} \frac{1-\lambda^n}{1-\lambda}$.
   Thus the Hilbert function of $M_m \otimes M_n$ is given by
  $$H(M_m \otimes M_n,j) =
 \begin{cases}
            j+1,          & \text{if } 0 \leq j \leq m-1;\\
            m,           & \text{if } m-1\leq j \leq n-1;\\
            m+n-j-1, & \text{if } n-1 \leq j \leq n+m-2;\\
            0,             &\text{otherwise}.
   \end{cases}$$

\begin{proposition}
  Let $1 \leq m \leq n$.  For $r=0,1,2,\dots,m-1$ there exists an element $\omega_r \in \MZ$ such that
  $\Delta(\omega_r)=0$ with $\Deg(\omega_r) =m+n-r-2$ and $\gr(\omega_r) = \sum_{i=0}^{r}  (-1)^{i+1}s^{m-1-i} t^{n-r+i-1}$.
\end{proposition}

\begin{proof}
  We begin by showing  that the homomorphism of $\Z$-modules $$\Delta : \FF_q(\MZ) \to \FF_{q+1}(\MZ)$$
   is surjective for all
   $q= n-1, n, \dots, m+n-2$.
We do this using downward induction on $q$.  For $q=m+n-2$ the codomain $\FF_{m+n-1}(\M)=0$
 and so the result is trivially true.

  Now suppose that $\Delta : \FF_q(\MZ) \to \FF_{q+1}(\MZ)$ is surjective and consider the map
  $\Delta : \FF_{q-1}(\MZ) \to \FF_{q}(\MZ)$.
  It is easy to verify that $(s+t) \sum_{k=0}^{m-i-1} (-1)^{k} s^{i+k} t^{q-i-k-1} = s^i t^{q-i}$ in $\M$
   for $q-n+1 \leq i \leq m-1$. 
   Therefore $\gr \Delta( \sum_{k=0}^{m-i-1} (-1)^{k} s^{i+k} t^{q-i-k-1}) = s^i t^{q-i}$.
   By the induction hypothesis, $\Delta(\FF_{q-1}(\MZ)) \supseteq \Delta(\FF_{q}(\MZ)) = \FF_{q+1}(\MZ)$.
   Furthermore,  $$\FF_{q}(\MZ) =  (\oplus_{i=q-n+1}^{m-1} \Z s^i t^{q-i}) \oplus \FF_{q+1}(\MZ)\ .$$
   Thus $\Delta(\FF_{q-1}(\MZ)) = \FF_{q}(\MZ)$ as claimed.


  Put $\omega'_r := \sum_{k=0}^{r}  (-1)^{k}s^{m-r+k-1} t^{n-k-1} \in \FF_{m+n-r-2}(\MZ)$.  It is easy to check that
   $(s+t)\omega'_r = 0$ and thus $\Deg(\Delta(\omega'_r)) \geq \Deg(\omega'_r) + 2$.
   Thus $\Delta(\omega'_r)) \in \FF_{m+n-r}(\MZ)$.  By the above, there exists
   $\omega''_r \in \FF_{m+n-r-1}(\MZ)$ such that $\Delta(\omega''_r) = -\Delta(\omega'_r)$.
   Taking $\omega_r := \omega'_r + \omega''_r \in \MZ$ we have
   $\Delta(\omega_r)=0$ with $\gr(\omega_r) = \omega'_r$ as required.
\end{proof}

\begin{remark}
  Note that $M_q^{\Z} =$  (the kernel of $\Delta:M_q \to M_q$) is one dimensional for all $q \geq n-1$.
  Since the direct sum decomposition of $\M$ into indecomposables has $m$ summands, this implies that
  the kernel of $\Delta:\M \to \M$ has dimension $m$ and thus $\{\omega_0,\omega_1,\dots,\omega_{m-1}\}$
  is a basis for this kernel.  Furthermore, this kernel is contained in $\FF_{n-1}(\M)$.
\end{remark}

\begin{theorem}\label{integral decomposition}
  There exist elements $\alpha_0, \alpha_1,\dots,\alpha_{m-1} \in \M$ such that
  for all $r=0,1,\dots,m-1$ we have
  \begin{enumerate}
    \item $\Deg( \Delta^j(\alpha_i) ) = i+j$ for all $0\leq i \leq r$ and $0 \leq j \leq m+n-2i-2$.\label{a}
   \item $\{ \gr(\Delta^j(\alpha_i)) \mid i \leq r, j \leq m+n-2i-2, i+j \leq r \}$ is linearly independent.\label{b}
   \item $\{ \Delta^j(\alpha_i) \mid 0 \leq i \leq r, m+n-i-r-2 \leq j \leq m+n-2i-2 \}$ is a basis for $\FF_{m+n-r-2}(\M)$.\label{c}
    \item $\Delta^{m+n-2r-2}(\alpha_r) = \omega_r$ and $\ell(\omega_r) = m+n-2r-1$.\label{d}
  \end{enumerate}
\end{theorem}

\begin{proof}
  We proceed by complete induction on $r$.
  For $r=0$ we take
     $\alpha_0 = 1/{\binomial{m+n-2}{m-1}}$.
     Then $\Delta^{m+n-2}(\alpha_0) = s^{m-1}t^{n-1} = \omega_0$.   Clearly this implies that
     $\ell(\omega_0) = m+n-1$.   It is also clear that $\{\omega_0\}$ is a basis for the
     one dimensional space $\FF_{m+n-2}(\M)$.
     Since $\Deg(\Delta^{m+n-2}(\alpha_0))=m+n-2$ we must have
     $\Deg(\Delta^{j}(\alpha_0))=j$ for all $0 \leq j \leq m+n-2$.  Finally, $\{\gr(\alpha_0)\} = \{\alpha_0\}$ is linearly independent.

     Assume then that the four assertions hold for all values less than or equal to $r$ and consider these
   four assertions for the value $r+1$.
  By the Clebsch-Gordan formula, $M_m \otimes M_n$ contains a summand isomorphic to
 $M_{m+n-2r-3}$.  Thus there exists $\omega \in \ker \Delta$ with $\ell(\omega)=m+n-2r-3$.
 Take $\alpha$ such that $\Delta^{m+n-2r-4}(\alpha) = \omega$.
 Since $\ell(\omega_k) > m+n-2r-3$ for all $k \leq r$ we may write 
 $\omega = \sum_{k=r+1}^{m-1} c_k\omega_k$ for some $c_k \in \Q$.  Therefore $\Deg(\omega) \leq \Deg(\omega_{r+1}) = m+n-r-3$.
 This implies that $\Deg(\alpha) \leq r+1$.

   Combining (\ref{a}) and (\ref{b}) we see that  $\{ \gr(\Delta^j(\alpha_i)) \mid i \leq r, i+j \leq r \}$ is a basis for
   $\oplus_{d=0}^{r} (\M)_d$.  Therefore we may write
   $\alpha = \sum_{i+j \leq r} a_{ij} \Delta^j(\alpha_i) + \alpha_{r+1}$ where $\alpha_{r+1} \in \FF_{r+1}(\M)$ and
   $a_{ij} \in \Q$ for all $i,j$.

   Now
   \begin{align*}
   0 &= \Delta(\omega) = \Delta^{m+n-2r-3}(\alpha)\\
       &=  \sum_{i+j \leq r} a_{ij} \Delta^{m+n-2r+j-3}(\alpha_i) + \Delta^{m+n-2r-3}(\alpha_{r+1})
 \end{align*}
 We consider this last expression in each degree $d=0,1,\dots,m+n-r-3$.
 The component of this expression in degree $d$ for $d=0,1,\dots,m+n-2r-4$ is trivially 0.
 In degree $d=m+n-2r-3$ we find only $a_{00} \Delta^{m+n-2r-3}(\alpha_0)$ and thus $a_{00}=0$.
 In degree $d=m+n-2r-2$ we find $a_{01} \Delta^{m+n-2r-3}(\alpha_1) + a_{10} \Delta^{m+n-2r-2}(\alpha_0)$.
 Therefore (using (\ref{c})) we have $a_{01}=a_{10}=0$.
 Continuing in this manner up to degree $d=m+n-r-3$ we find that $a_{ij}=0$ for all $i,j$.
 Therefore $\alpha = \alpha_{r+1}$ and $\Deg(\alpha) \geq r+1$.  We already observed that
 $\Deg(\alpha) \leq r+1$ and therefore $\Deg(\alpha_{r+1})=r+1$.
 Since $\Deg(\Delta^{m+n-2r-4}(\alpha_{r+1}))=m+n-r-3$ we must have
 $\Deg(\Delta^{j}(\alpha_{r+1}))=j+r+1$ for all $j=0,1,\dots,m+n-2r-4$ which proves (\ref{a}).
 In particular, $\Deg(\omega) \geq m+n-r-3$ and so we must have $\omega=c_{r+1}\omega_{r+1}$.
Take $\alpha_{r+1} = c_{r+1}^{-1}\alpha$.  Then  
 $\Delta^{m+n-2r-4}(\alpha_{r+1}) = \omega_{r+1}$ which proves (\ref{d}).

   Now $\{\Delta^j(\alpha_i) \mid 0 \leq i \leq r+1, 0 \leq j \leq m+n-2i-2\}$ is a basis for
   $\oplus_{i=0}^{r+1} M_{m+n-2i-1}$ and so in particular is linearly independent.
   Counting dimensions this implies that
   $\{\Delta^j(\alpha_i) \mid i+j \geq m+n-r-1, 0 \leq i \leq r+1, 0 \leq j \leq m+n-2i-2\}$ is a basis for
   $\FF_{m+n-r-3}(\M)$ which proves (\ref{c}).

   Finally we prove (\ref{b}).  Assume there exists a linear relation
   $$\sum_{i+j=d} b_{ij} \gr(\Delta^j(\alpha_i)) = 0$$
    with scalars $b_{ij} \in \Q$ where $d \leq r+1$.
  This linear relation (together with (\ref{a})) implies that
  $\sum_{i+j=d} b_{ij} \Delta^{m+n-r-d+j-3}(\alpha_i)) \in \FF_{m+n-r-2}(\M)
  = \Span_{\Q}\{\Delta^k(\alpha_i) \mid 0 \leq i \leq r, i+k \geq m+n-r-2, k \leq m+n-2k-2\}$.
  But by (\ref{c}) this means we may write
  $\sum_{i+j=d} b_{ij} \Delta^{m+n-r-d+j-3}(\alpha_i)) = \sum_{i+j > m+n-3} b'_{ij} \Delta^j(\alpha_i)$.
  But we have already seen that  $\{\Delta^j(\alpha_i) \mid i+j \geq m+n-r-1, 0 \leq i \leq r+1, 0 \leq j \leq m+n-2i-2\}$
  is linearly independent.  Therefore each $b_{ij}=0$ which proves (\ref{b}).
  \end{proof}

\subsection{Decomposing $V_m \otimes V_n$}
Next we want to determine an explicit decomposition of a tensor product of indecomposable $C_p$-modules, $V_m \otimes V_n$.
Proposition~\ref{decomp} gives an abstract decomposition of $V_m \otimes V_n$.
 We want to obtain a more explicit description of this decomposition.
 To do this we consider the integer lattices $M_m(\Z)$ and $M_n(\Z)$
 and the surjection $\rho: M_m(\Z) \otimes M_n(\Z) \to \V$ given by reduction modulo the prime $p$.
 
   The following well-known result and its proof are included for the reader's convenience.
   \begin{lemma}\label{lattice result}
      Let $U$ be an $n$ dimensional $\Q$ vector space $U \cong \Q^n$ and suppose $W$ is an 
      $r$ dimensional subspace of $U$.  Let $U(\Z) = \Z^n$ be the natural lattice in $U$.  Let $p$ be prime
      and let $\rho$ denote the reduction modulo $p$ map.
      Then $W(\Z) := W \cap U(\Z)$ is a rank $r$ lattice and $\rho(W(\Z)) \cong \F_p^r$. 
   \end{lemma}
   
   \begin{proof}
     The lattice $K_0 := p U(\Z)$ is the kernel of the map $\rho$.   
     $W(\Z) \cap K_0$ is a free abelian group whose rank equals its minimal number of generators. 
      Choose a vector space basis $\{v_1,v_2,\dots,v_r\}$ of $W$.
      Scaling the $v_i$ we may suppose that $v_i \in K_0$ and $u_i := (1/p)v_i \in U(\Z) \setminus K_0$.
        Thus $W(\Z) \cap K_0$ has rank at least $r$.    If $W(\Z) \cap K_0$ required more than $r$ generators we would find a relation
      among them from the linear $\Q$ linear dependence among them.
     Thus the rank of the lattice $W(\Z) \cap K_0$ is $r$ and this lattice is generated by $v_1,v_2,\dots,v_r$.
   
     Furthermore $\{u_1,u_2,\dots,u_r\}$ is a basis of $W(\Z)$.  
     To see this, take any $w \in W(\Z)$.  Then 
     $pw \in W(\Z) \cap K_0$ and so we may write $pw = \sum_{i=1}^r c_i v_i$ where each $c_i \in \Z$.  
    Then $w = \sum_{i=1}^r c_i u_i$.
    This implies that the index of $W(\Z) \cap K_0$ in $W(\Z)$ is $p^r$ and 
    $\rho(W(\Z)) \cong W(\Z) / (W(\Z) \cap K_0) \cong (\Z/p\Z)^r$. 
%
%
%
%
%
   \end{proof}

\begin{theorem}
 Suppose  $1 \leq m \leq n \leq p$ with $m+n \geq p+1$.
 Then  $\ell(\rho(\omega_r))=p$ for all $r=0,1,\dots,m+n-p-1$.
\end{theorem}

\begin{proof}
   By Proposition~\ref{decomp}, the kernel of $\Delta$ on $V_m \otimes V_n$ is an $m$ dimensional
   $\field_p$ vector space.  Since $\{\rho(\omega_0),\rho(\omega_1),\dots,\rho(\omega_m)\}$ is a 
   linearly independent subset of $\ker \Delta$, it must be a basis for $\ker \Delta$.
  Thus
  \begin{align*}
   \Delta^{p-1}(\V) &\subseteq \ker \Delta \cap \FF_{p-1}(\V)\\
                             &= \Span_{\field_p}\{\rho(\omega_0),\rho(\omega_1),\dots,\rho(\omega_{m}\} \cap \FF_{p-1}(\V)\\  
                               & = \Span_{\field_p}\{\rho(\omega_0),\rho(\omega_1),\dots,\rho(\omega_{m+n-p-1})\}.
   \end{align*}
     By Proposition~\ref{decomp}, $\Delta^{p-1}(\V)$ has dimension $m+n-p$ which implies that
 $\Delta^{p-1}(\V) = \Span_{\field_p}\{\rho(\omega_0),\rho(\omega_1),\dots,\rho(\omega_{m+n-p-1})\}$.
 In particular $\ell(\rho(\omega_r))=p$ for $r=0,1,\dots,m+n-p-1$.
\end{proof}

\begin{proposition}
   There exist $\b{\beta}_0,\b{\beta}_1,\dots,\b{\beta}_{m+n-p-1} \in \V$ such that
  $\Delta^{p-1}(\b{\beta}_r) = \rho(\omega_r)$ and $\Deg(\b{\beta}_r)=m+n-p-r-1$ for $r=0,1,\dots,m+n-p-1$ .
\end{proposition}

\begin{proof}
  By the above theorem there must exist $\b{\beta'}_0,\b{\beta'}_1,\dots,\b{\beta'}_{m+n-p-1} \in \V$ such that
  $\Delta^{p-1}(\b{\beta'}_r) = \rho(\omega_r)$ for $r=0,1,\dots,m+n-p-1$.
  Since $\Deg(\rho(\omega_r)) = m+n-r-2$ we have $\Deg(\b{\beta'}_r) \leq m+n-p-1-r$ for $r=0,1,\dots,m+n-p-1$.
  This implies that $\Deg(\b{\beta'}_{m+n-p-1})=0$ and so we may take $\b{\beta}_0=\b{\beta'}_0$.

  Assume, by downward induction, that we have chosen $\b{\beta}_i$ with $\Delta^{p-1}(\b{\beta}_i) = \rho(\omega_i)$
  and $\Deg(\b{\beta}_i)=m+n-p-i-1$ for $i=r+1,r+2,\dots,m+n-p-1$ (where $r \geq 0$).
  The set $A_{r+1} := \{\Delta^j(\b{\beta}_i) \mid r+1 \leq i \leq m+n-p-1, 0 \leq j \leq i-r-1\}$ is linearly independent
  and consists of elements $x$ satisfying $\Deg(x) \leq r+1$.
  Since the cardinality of $A_{r+1}$ is
  $\sum_{i=r+1}^{m+n-p-1} (i-r) = \binomial{m+n-p-1-r}{2} = \dim\left((V_m \otimes V_n) / \FF_{m+n-p-r-1}(V_m \otimes V_n)\right)$,
  it follows that the natural image of $A_{r+1}$ forms a basis for $(V_m \otimes V_n) / \FF_{m+n-p-r-1}(V_m \otimes V_n)$.
  Choose $\b{\gamma}$ with $\Deg(\b{\gamma})=m+n-p-r-1$ such that the set
  $\{\b{\gamma}\} \sqcup \{\Delta^j(\b{\beta}_i) \mid r+1 \leq i \leq m+n-p-1, 0 \leq j \leq i-r\}$ similarly yields a basis for
  $(V_m \otimes V_n) / \FF_{m+n-p-r}(V_m \otimes V_n)$.
  Write $\b{\beta'}_r = c_0 \b{\gamma} + \sum_{i=r+1}^{m+n-p-1} \sum_{j=0}^{i-r} c_{ij}\Delta^j(\b{\beta}_i) + \b{\gamma'}$
  where $c_0,c_{ij} \in \F_p$ and $\b{\gamma'} \in \FF_{m+n-p-r}(V_m \otimes V_n)$.
  Then $\rho(\omega_r) = \Delta^{p-1}(\b{\beta'}_r) = \Delta^{p-1}(c_0 \b{\gamma}) +  \sum_{i=r+1}^{m+n-p-1} c_{i0} \rho(\omega_{m+n-p-i-1}) + \Delta^{p-1}(\b{\gamma'})$
  where $\Deg(\Delta^{p-1}(\b{\gamma'})) > \Deg(\rho(\omega_{m+n-p-i-1}))$ for all $i \geq r+1$.
  Therefore $c_{i0}=0$ for all $i=r+1,r+2,\dots,m+n-p-1$ and $\rho(\omega_r) = \Delta^{p-1}(c_0 \b{\gamma} + \b{\gamma'})$.
  Setting $\b{\beta}_r = c_0 \b{\gamma} + \b{\gamma'}$ yields $\Deg(\b{\beta}_r)=m+n-p-r-1$ and $\Delta^{p-1}(\b{\beta}_r)=\rho(\omega_r)$ as required.
\end{proof}


%

 We have seen that $M_m \otimes M_n \cong \bigoplus_{r=0}^{m-1} M_{m+n-2i-1}$
 and that we may arrange this decomposition so that the socle of the summand $M_{m+n-2i+1}$ is
 spanned by $\omega_i$.  Furthermore $\alpha_i$ is a generator of the summand $M_{m+n-2i+1}$  
 with $\Deg(\alpha_i)=i$ and $\Delta^{m+n-2i}(\alpha_i) = \omega_i$ for $i=0,1,\dots,m-1$
 Moreover, by clearing denominators, we may assume that  $\alpha_i \in (M_m \otimes M_n)(\Z)$ with
 $\rho(\alpha_i) = a_i \omega_i \ne 0$ for some $a_i \in \Z$.


 \begin{theorem}\label{modular decomposition}
 Suppose  $1 \leq m \leq n \leq p$ with $m+n \geq p+1$.
   For $m+n-p \leq r \leq m-1$,
   $\ell(\rho(\omega_r))=m+n-2r-1$.
   Furthermore, $$\V = \bigoplus_{r=m+n-p}^{m-1} V_{m+n-2r-1} \bigoplus (m+n-p)V_p$$ where
  $$V_{m+n-2r-1} = \Span_{\field_p}\{\Delta^j(\rho(\alpha_{r})) \mid 0 \leq j \leq m+n-2r-2\}$$ for
   $m+n-p \leq r \leq m-1$.
   In particular $$\rho(M_{m+n-2r-1}(\Z)) = V_{m+n-2r-1}$$ for $m+n-p \leq r \leq m-1$
   where $M_{m+n-2r-1}$ and $V_{m+n-2r-1}$ are indecomposable summands of $\M$ and $\V$
   generated by $\alpha_r$ and $\rho(\alpha_r)$ respectively.
 \end{theorem}

 \begin{proof}
    It suffices to show that $\Delta^{m+n-2r-2}(\rho(\alpha_r)) \neq 0$ for all $r \geq m+n-p$.  Fix $r \geq m+n-p$.
       Let  $\KM{r}$ denote the kernel of $\Delta^{m+n-2r-2}:M_m\otimes M_n \to M_m\otimes M_n$.
   Then $\KM{r}$ is a $\Q$ vector space and we write $\KM{r}(\Z) := \KM{r} \cap (M_m(\Z) \otimes M_n(\Z))$.
   Observe that the set
   \begin{align*}
       \{ \Delta^j&(\alpha_i) \mid 0 \leq i \leq r, 2r-2i+1 \leq j \leq m+n-2i-2\}\\
         & \sqcup \{ \Delta^j(\alpha_i) \mid r+1 \leq i \leq m-1, 0 \leq j \leq m+n-2i-2\}
  \end{align*}
    is a basis for $\KM{r}$.  Either from this or by Equation~(\ref{explicit CG for M}) we have
    \begin{align*}
     \dim_{\Q} \KM{r} 
                                 &= \sum_{i=0}^r (m+n-2r-2) + \sum_{i=r+1}^{m-1} (m+n-2i-1)\\
                                 &= (r+1)(m+n-2r-2) + \sum_{i=r+1}^{m-1} (m+n-2i-1).
   \end{align*}
     Let $\KV{r}$ denote the kernel of $\Delta^{m+n-2r-2}:V_m\otimes V_n \to V_m\otimes V_n$.
     By Proposition~\ref{decomp}, we see that
     $\dim_{\F_p} \KV{r} =  (r+1)(m+n-2r-2) + \sum_{i=r+1}^{m-1} (m+n-2i-1) = \dim_{\Q} \KM{r}$.
     Therefore, applying Lemma~\ref{lattice result}, we see that $\rho(\KM{r}(\Z)) = \KV{r}$.  
     Since $\alpha_r \notin \KM{r}$, this implies that
     $\rho(\alpha_r) \notin \KM{r}$, i.e., $\Delta^{m+n-2r-2}(\rho(\alpha_r)) \ne 0$ as required.
      Thus
     $\Delta^{m+n-2r-2}(\rho(\alpha_r))$ is a non-zero multiple of $\rho(\omega_r)$ and so
     $\ell(\rho(\omega_{r})) \geq m+n-2r-1$.
     Since this is true for all $r=m+n-p, m+n-p+1,\dots,m-1$, comparing with Proposition~\ref{decomp},
     shows that $\ell(\rho(\omega_{r})) = m+n-2r-1$ as required.
 \end{proof}

\begin{remark}\label{compatible}
Since $\rho(\Delta^{m+n-2r-2}(\alpha_r)) \neq 0$ we may replace $\alpha_r$ by an integer multiple
  of itself in order to arrange that $\Delta^{m+n-2r-2}(\rho(\alpha_r)) = \rho(\omega_r)$
  for $r=m+n-p,m+n-p+2,\dots,m-1$.
\end{remark}

\begin{remark}
  One component of our proofs of Theorems~\ref{integral decomposition} and \ref{modular decomposition} involves showing that the multiplication maps
  $$(s+t)^{m+n-2r-2}\cdot  : (M_m \otimes M_n)_r \to (M_m \otimes M_n)_{m+n-r-2}$$ for $r=0,1,\dots,m-1$ and the maps
  $$(s+t)^{m+n-2r-2}\cdot  : (V_m \otimes V_n)_r \to (V_m \otimes V_n)_{m+n-r-2}$$ for $r=m+n-p,m+n-p+1,\dots,m-1$
  are surjective.    Another way to show this step is to consider the matrix associated to these maps with respect to the
  basis of monomials in $s$ and $t$.  This matrix is given by
 $$D_{r+1}(m+n-2r-2,m-r-1):=\begin{pmatrix}  \binomial {m+n-2r-2} {m-r-1+i-j}
  \end{pmatrix}_{1\leq i \leq r+1 \atop 1\leq j \leq  r+1}\ .
  $$
  \name{Srinivasan} \cite{Sr} shows that this matrix is row equivalent to his {\em Pascal matrix}
 $$P_{r+1,r+1}(m+n-2r-2,m-r-1):=\begin{pmatrix} \binomial {m+n-2r-3+i} {m-r-2+j}
  \end{pmatrix}_{1\leq i \leq r+1 \atop 1\leq j \leq  r+1}\ .
  $$
  Moreover this row equivalence may be obtained using only determinant preserving row operations.
   \name{Srinvasan} shows that this later matrix has determinant
   $$ \frac{1!\, 2! \cdots r!}{(m-r)^{r} (m-r+1)^{r-1} \cdots (m-1)}
   \prod_{c=0}^r \binomial{m+n-2r-2+c}{m-r-1}\ . $$
 This determinant is always non-zero and is non-zero modulo $p$ if and only if $m+n-r-2 < p$.
\end{remark}

\begin{theorem}
   Suppose $1 \leq m_i \leq p$ for $i=1,2,\dots,r$.
   Write $V_{n_1} \otimes V_{n_2} \otimes \dots \otimes V_{n_r} \cong \oplus_{i=1}^p a_i V_i$.
   Then $M_{n_1} \otimes M_{n_2} \otimes \dots \otimes M_{n_{r}}$ contains
   a summand $N$ with $N\cong \oplus_{i=1}^{p-1} a_i M_i$ such that $\rho(N(\Z))=W$
   where $W$  is a summand of  of $\otimes_{k=1}^{n-1} V_{n_k}$ with
   $W \cong \oplus_{i=1}^{p-1} a_i V_i$.   More explicitly, we may decompose $N$ and $W$
   into indecomposables summands $N = \oplus_{\alpha \in \Gamma} N_\alpha$ and
   $W = \oplus_{\alpha \in \Gamma} W_\alpha$ with $\dim_\Q M_\alpha = \dim_{\fp} W_\alpha$ and
   $\rho(M_\alpha(\Z))=W_\alpha$ for all $\alpha \in \Gamma$.
\end{theorem}

\begin{proof}
  The proof is by induction on $r$.  The result is trivial for $r=1$.

  Decompose $\otimes_{k=1}^{r-1} V_{n_k}$ into a direct sum of indecomposables $C_p$-modules:
   $$V_{n_1} \otimes V_{n_2} \otimes \dots \otimes V_{n_{r-1}} =  \oplus_{\alpha \in A} W_\alpha.$$
Define $A' := \{ \alpha \in A \mid \dim W_\alpha < p\}$ and
$A'' := A \setminus A' = \{ \alpha \in A \mid \dim W_\alpha = p\}$.
  Define $W' :=  \oplus_{\alpha\in A'} W_\alpha$ and $W'' :=  \oplus_{\alpha\in A''} W_\alpha$ so that
  $\otimes_{k=1}^{r-1} V_{n_k} = W' \oplus W''$.

 By induction
  $M_{n_1} \otimes M_{n_2} \otimes \dots \otimes M_{n_{r-1}}$ contains a summand $U'$ with
   $U' = \oplus_{\alpha \in A'} N_\alpha$ where each $N_\alpha \cong M_{\theta(\alpha)}$
   with $\theta(\alpha) = \dim_{\Q} N_\alpha = \dim_{\field_p} W_\alpha < p$ and such that
   $\rho(N_\alpha(\Z)) = W_\alpha$ for all $\alpha \in A'$.  Thus $\rho(U'(\Z))=W'$.

  Decompose  $W_\alpha \otimes V_{n_r} = \oplus_{\beta \in B_\alpha} W_{\alpha,\beta}$ and
  define $B'_\alpha := \{\beta \in B_\alpha \mid \dim W_{\alpha,\beta} < p\}$ and
   $B''_\alpha := B_\alpha \setminus B'_\alpha$.
 By Theorem~\ref{modular decomposition} and Remark~\ref{compatible},
  $M_\alpha \otimes M_{n_r}$ contains a summand
 $\oplus_{\beta\in B'_\alpha} N_{\alpha,\beta}$
 with $N_{\alpha,\beta} \cong M_{\theta(\beta)}$ where $\theta(\beta) = \dim_{\Q} N_{\alpha,\beta}
  = \dim_{\field_p} W_{\alpha,\beta} < p$ and such that
 $\rho(N_{\alpha,\beta}(\Z))=W_{\alpha,\beta}$  for all $\beta \in B'_\alpha$ and all $\alpha \in A'$.

   Thus we have
  \begin{align*}
  \bigotimes_{k=1}^r V_{n_k} &\cong (W' \otimes V_{n_r}) \oplus (W'' \otimes V_{n_r})\\
           &= (\oplus_{\alpha \in A'} W_\alpha \otimes V_{n_r}) \oplus(W'' \otimes V_{n_r})\\
        & = (\oplus_{\alpha \in A'} \oplus_{\beta \in B_\alpha}  W_{\alpha,\beta}) \oplus (W'' \otimes V_{n_r})\\
                  &= (\oplus_{\alpha \in A'} \oplus_{\beta \in B'_\alpha}  W_{\alpha,\beta})
                  \oplus  (\oplus_{\alpha \in A'} \oplus_{\beta \in B''_\alpha}  W_{\alpha,\beta}) \oplus (W'' \otimes V_{n_r})
  \end{align*}
   where $ (\oplus_{\alpha \in A'} \oplus_{\beta \in B''_\alpha}  W_{\alpha,\beta}) \oplus (W'' \otimes V_{n_r})$ is a free
   $C_p$-module and $W \cong \oplus_{\alpha \in A'} \oplus_{\beta \in B'_\alpha}  W_{\alpha,\beta}$.

   Taking $N$ to be the summand $N := \oplus_{\alpha \in A'} \oplus_{\beta \in B'_\alpha}  N_{\alpha,\beta}$
   of $\otimes_{k=1}^r M_{n_k}$ we have $\rho(N_{\alpha,\beta}(\Z))=W_{\alpha,\beta}$ for
   all $\alpha \in A'$ and all $\beta \in B'_\alpha$  and $\rho(N(\Z))=W$ as required.
  \end{proof}

\begin{corollary}\label{key}
  Suppose $1 \leq n_1,n_2,\dots,n_k \leq p$.
  Every invariant $f \in (\otimes_{k=1}^r V_{n_k})^{C_p}$ may be expressed as a sum $f=f_0 + f_1$ where
  $f_0$ is integral (i.e., $f_0 = \rho (F_0)$ for some $F_0 \in  (\otimes_{k=1}^r M_{n_k}(\Z))^{\Z})$ and $f_1$ is a transfer.
\end{corollary}

We now apply this to symmetric algebras.
\begin{theorem}\label{main}
  Let $1 < n_1, n_2, \dots n_r \leq p$ and $0 \leq d_1, d_2, \dots, d_r \leq p-1$.
  Every invariant $f \in \field_p[V_{n_1} \oplus V_{n_2} \oplus \dots \oplus V_{n_r}]_{(d_1,d_2,\dots,d_r)}^{C_p}$
  may be written as $f' + f''$ where $f'$ is integral and $f''$ is a transfer,
  i.e., $f' = \rho(F')$ for some $F' \in \Z[M_{n_1} \oplus M_{n_2} \oplus \dots \oplus M_{n_r}]^{\Z}_{(d_1,d_2,\dots,d_r)}$
  and $f'' = \tr^{C_p}(F'')$ for some $F'' \in \fp[V_{n_1} \oplus V_{n_2} \oplus \dots \oplus V_{n_r}]_{(d_1,d_2,\dots,d_r)}$.
\end{theorem}

\begin{proof}
  Let $d < p$.
  The symmetric group on $d$ letters, $\Sigma_d$, acts on $\otimes^d V^*_n$ by permuting factors.
Furthermore $\Sym^d V^*_n =  (\otimes^d V^*_n)^{\Sigma_d}$.  Since $d < p$, the group $\Sigma_d$ is non-modular
and therefore $\Sym^d V^*_n$ has a $\Sigma_d$-stable complement:
$\otimes^d V^*_n = \Sym^d V^*_n \oplus U$.
Since the actions of $C_p$ (in fact all of $\GL(V^*_n)$) and $\Sigma_d$ commute, the complement $U$ is also a $C_p$-module
(in fact a $\GL(V^*_n)$-module).
Therefore $\Sym^d V^*_n$ is a summand of $\otimes^d V^*_n$ as a $C_p$-module.

  Similarly $\Sym^d M^*_n$ is a summand of the $\Z$-module $\otimes^d M^*_n$.
  The projection of $\otimes^d V^*_n$ onto $\Sym^d V^*_n$ is given by the
Reynolds operator $\Pi_{\Sigma_d} = \frac{1}{d!}\sum_{\tau \in \Sigma_d} \tau$.
The same formula give the projection of $\otimes^d M^*_n$ onto $\Sym^d M^*_n$.

   In the same manner we see that
  $\field_p[V_{n_1} \oplus V_{n_2} \oplus \dots \oplus V_{n_r}]_{(d_1,d_2,\dots,d_r)}
= \Sym^{d_1} V^*_{n_1} \otimes \Sym^{d_2} V^*_{n_2}  \otimes \dots \otimes \Sym^{d_r} V^*_{n_r}$ is a summand of the
$C_p$-module
\hbox{$\otimes_{i=1}^r \otimes^{d_i} V^*_{n_i}$} and that
$\Q[M_{n_1} \oplus M_{n_2} \oplus \dots \oplus M_{n_r}]_{(d_1,d_2,\dots,d_r)}
= \Sym^{d_1} M^*_{n_1} \otimes \Sym^{d_2} M^*_{n_2}  \otimes \dots \otimes \Sym^{d_r} M^*_{n_r}$ is a summand of the
$\Z$-module  $\otimes_{i=1}^r\otimes^{d_i} M^*_{n_i}$.
The projection onto these summands is given by the Reynolds operator $\Pi$ associated to Young subgroup
$\Sigma_{d_1,d_2,\dots,d_r} := \Sigma_{d_1} \times \Sigma_{d_2} \times \dots \times \Sigma_{d_r}$ where
$$\Pi = \Pi_{\Sigma_{d_1,d_2,\dots,d_r}} =
      \frac{1}{d_1! d_2! \cdots d_r!} \sum_{\tau \in \Sigma_{d_1,d_2,\dots,d_r} } \tau\ .
   $$

By Corollary~\ref{key}, every invariant
$f \in \field_p[V_{n_1} \oplus V_{n_2} \oplus \dots \oplus V_{n_r}]_{(d_1,d_2,\dots,d_r)}^{C_p}$  can be written
as a sum $f = f_0 + f_1$ where
 $f_0 =  \rho(F_0)$ for some $F_0 \in (\otimes_{j=1}^r \otimes^{d_j} M^*_{n_j}(\Z))^{\Z}$ and
$f_1 = \tr^{C_p}(F_1)$ for some $F_1 \in \otimes_{j=1}^r \otimes^{d_j} V^*_{n_j}$.
Therefore
\begin{align*}
  f &= \Pi(f) = \Pi(f_0 + f_1)=\Pi(f_0) + \Pi(f_1)\\
     & = \Pi(\rho(F_0)) + \Pi(\tr^{C_p}(F_1))\ .
\end{align*}
  Clearly $\Pi(\rho(F_0)) = \rho(\Pi(F_0))$.  Since the action of $\Sigma_{d_1,d_2,\dots,d_r}$ and $C_p$ on
  $\otimes_{i=1}^r \otimes^{d_i} V^*_{n_i}$ commute, we have
  $\Pi(\tr^{C_p}(F_1))=\tr^{C_p}(\Pi(F_1))$.
  Similarly the actions of $\Sigma_{d_1,d_2,\dots,d_r}$ and $\Z$ on $\otimes_{i=1}^r \otimes^{d_i} M^*_{n_i}$ commute and thus
   $\Pi(F_0)$ is a $\Z$ invariant since $F_0$ is.
  Therefore $f = \rho(\Pi(F_0)) + \tr^{C_p}(\Pi(F_1))$ where
  $ \Pi(F_0) \in \Q[M_{n_1} \oplus M_{n_2} \oplus \dots \oplus M_{n_r}]^{\Z}_{(d_1,d_2,\dots,d_r)}$
  and $\Pi(F_1) \in \fp[V_{n_1} \oplus V_{n_2} \oplus \dots \oplus V_{n_r}]_{(d_1,d_2,\dots,d_r)}$.
  Hence we have written $f$ as the sum of a integral invariant and a transfer.

Also note that Roberts' isomorphism implies that $\Pi(F_0) = \psi(h)$ for some
$h \in \C[R_1 \oplus R_{n_1-1} \oplus R_{n_2-1} \oplus \dots \oplus R_{n_r-1}]^{\SL_2(\C)}$.
\end{proof}

  Combining Therem~\ref{main} with the Periodicity Theorem we have a proof of the conjecture:
  \begin{theorem}\label{conj proved}
     Let $V = \oplus_{i=1}^r V_{n_i}$.  For each $i=1,2,\dots,r$, choose a generator $z_i$ of the cyclic module $C_p$-module
     $V^*_{n_i}$, i.e., choose $z_i \in V^*_{n_i} \setminus \Delta(V^*_{n_i})$.  Put $N_i := \norm(z_i)$.  Then
     $\field_p[V]^{C_p}$ is generated by $N_1,N_2,\dots,N_r$ together with a finite set of integral invariants and a finite
     set of transfer invariants.
  \end{theorem}

  \begin{proof}
Given $f \in \field_p[V]^{C_p}$ we may use the decomposition from the Periodicity Theorem to write $f = f^\sharp + f^\flat$ with
$f^\sharp = \sum_{i=1}^r f_i N_i$ where each $f_i \in \field_p[V]^{C_p}$ and $f^\flat \in (\field_p[V]^{C_p})^\flat$
Thus we may choose a generating set for $\field_p[V]^{C_p}$ consisting of elements of  $(\field_p[V]^{C_p})^\flat$ together with
the $r$ norms $N_1,N_2,\dots,N_r$.  Of course, we can and will choose the elements $(\field_p[V]^{C_p})^\flat$ to be multi-graded.
Given such a generator $f \in (\field_p[V]^{C_p}_{(d_1,d_2,\dots,d_r)})^\flat$ we see by the Periodicity Theorem that if there
exists an $i$ with $d_i > p - n_i$ then $\ell(f) = p$, i.e, $f = \Delta^{p-1}(F)$ for some $F \in \field_p[V]^\flat$.  Since
$\Delta^{p-1}(F) = (\sigma-1)^{p-1}(F) = (1 + \sigma + \sigma^2 + \dots + \sigma^{p-1})(F) = \tr(F)$,
we see that $d_i > p - n_i$ forces $f$ to be in the image of the transfer.    Note that the degree conditions
$d_i \leq p-n_i$ ensures that the vector space spanned by the integral non-transfer invariants is finite dimensional.

  Following \cite[Theorem~6.2]{shank}, we see that there is a homogeneous system of parameters for $\field_p[V]^{C_p}$ consisting
  of $N_1,N_2, \dots, N_r$ together with a number of transfers of degree $p-1$.  Let $A$ denote the polynomial algebra generated by
  this homogeneous system of parameters.  Since $\field_p[V]$ is Cohen-Macaulay we have the Hironaka decomposition
  $$\field_p[V] = \oplus_{k=1}^q A h_k$$ where $h_k \in \field_p[V]$ for all $k$.  Since the transfer is an $A$-module map
  $\{\tr(h_k) \mid k=1,2,\dots,q\}$ forms a set of $A$-module generators for the ideal $\tr(\field_p[V])$.
  These $q$ transfers together with the $\dim V$ many elements in the homogeneous system of parameters and the
  finitely many integral non-transfer invariants form a finite algebra generating set for $\field_p[V]^{C_p}$.
  \end{proof}

  The following more explicit formulation of the above theorem is useful.
  \begin{corollary}\label{explicit key}
     Let $V = \oplus_{i=1}^r V_{n_i}$.  For each $i=1,2,\dots,r$, choose a generator $z_i$ of the cyclic module $C_p$-module
     $V^*_{n_i}$, i.e., choose $z_i \in V^*_{n_i} \setminus \Delta(V^*_{n_i})$.  Put $N_i := \norm(z_i)$
     and $z_{ij} = \Delta^j(z_i)$ for all $1\leq i \leq r$ and $0 \leq j \leq n_i-1$.
     Suppose there exist invariants $f_{ij} \in \field[V]^G$ and positive integers $d_{ij}$ such
     $\LT(f_{ij}) = z_{ij}^{d_{ij}}$  for all $1\leq i \leq r$ and $1 \leq j \leq n_i-1$.
      Put $d_{0j}=p$ (since $\LT(N_i) = z_i^p$).   Then
     $\field_p[V]^{C_p}$ is generated by the following invariants
     \begin{itemize}
     \item   $N_1,N_2,\dots,N_r$
     \item  $f_{ij}$ with  $1\leq i \leq r$ and $1 \leq j \leq n_i-1$
     \item a finite set of integral invariants
     \item
     $\displaystyle\tr\left(\prod_{i=1}^r \prod_{j=0}^{n_i-1} z_{ij}^{a_{ij}}\right) \text{ with } 0 \leq a_{ij} < d_{ij} \text{ for all } 1\leq i \leq r \text{ and }1 \leq j \leq n_i-1$.
 \end{itemize}
  \end{corollary}

  \begin{proof}
     The hypotheses imply (by \cite[Lemma~6.2.1]{CW}) that the set
    $$\{N_1,N_2,\dots,N_r\} \cup  \{f_{ij} \mid 1\leq i \leq r, 1 \leq j \leq n_i-1\}$$
     forms a homogeneous system of parameters.
       Let $A$ denote the polynomial algebra generated by these $\dim V$ many invariants.   By Theorem~\ref{conj proved}, $\field_p[V]^{C_p}$ is generated by
    $A$ together with a finite set of integral invariants and some transfers.
    Let $\Gamma$ denote the set of monomials
    $\Gamma = \{\prod_{i=1}^r \prod_{j=0}^{n_i-1} z_{ij}^{a_{ij}} \mid 0 \leq a_{ij} < d_{ij} \}$.
     Then
    $$\field_p[V] = \oplus_{\gamma\in\Gamma} A \gamma\ .$$
    Thus $\{\tr(\gamma) \mid \gamma \in \Gamma\}$ is a set of $A$-module generators for the
    ideal $\tr(\field[V])$.
  \end{proof}

 \section{Applications}
 \renewcommand{\arraystretch}{1.2}
  We use Corollary~\ref{explicit key} to give generators for the invariant ring of a number
  of representations of $C_p$.

\subsection{The Invariant Ring $\field[V_2 \oplus V_4]^{C_p}$}
  We mentioned in the introduction that the $C_p$ representation $V_2 \oplus V_4$ is the only remaining reduced representation
  whose ring of invariants is likely to be computable by the SAGBI basis method originally developed by Shank.
  Here we will  find generators for this ring
  much more easily by using the proof of the conjecture.

    We need to find the ring of covariants of $R_1 \oplus R_3$.
      A method to find generators for this ring is given in \cite[\S 138A]{G-Y}.  Letting
 $L$ denote the linear form and $f$ the cubic form we have the following 13 generators for this ring of covariants.
 \begin{center}
\begin{tabular}{lcclll}
   Covariant & Order & Bi-degree& LM & LM(Source)\\ \hline
  $L$ & 1 & (1,0)& $a_0x$& $x_1$\\
   $f $& 3 & (0,1) & $b_0x^3$&$x_2$\\
  $H := (f,f)^2$ & 2 & (0,2) &$b_1^2 x^2$ &$y_2^2$\\
  $T := (f,H)^1$  & 3 & (0,3) & $b_1^3 x^3$ &$y_2^3$\\
  $\Delta := (H,H)^2$  & 0 & (0,4) & $b_1^3 b_3$&$y_2^2 z_2^2$\\
  $(f,L)^1$  & 2 & (1,1) & $a_1 b_0 x^2$&$x_1 y_2$\\
  $(f,L^2)^2$  & 1 & (2,1) & $a_1^2 b_0 x$&$x_1^2 z_2$\\
   $(f,L^3)^3$  & 0 & (3,1) & $a_1^3 b_0$&$x_1^3 w_2$\\
  $(H,L)^1$  & 1 & (1,2) & $a_1 b_1^2 x$&$x_1 y_2 z_2$\\
  $(H,L^2)^2$  & 0 & (2,2) & $a_1^2 b_1^2$&$ x_1^2 z_2^2 $\\
  $(T,L)^1$  & 2 & (1,3) & $a_1 b_1^3 x^2$&$ x_1 y_2^2 z_2 $\\
  $(T,L^2)^2$  & 1 & (2,3) & $a_1^2 b_1^3 x$&$ x_1^2 y_2 z_2^2 $\\
   $(T,L^3)^3$  & 0 & (3,3) & $a_1^3 b_1^3$&$ x_1^3 z_2^3 $\\
   \\
   \multicolumn{5}{c}{Table 10.1. {\bf Covariants of }{\boldmath $R_1 \oplus R_3$}}\\
\end{tabular}
\end{center}
Here we are using $\{x,y\}$ as a basis for the dual of the  first copy of $R_1$,
 $\{a_0,a_1\}$ as a basis for the dual of the second copy of $R_1$ and
$\{b_0,3b_1,3b_2,b_3\}$ as the basis for $R_3^*$.
Thus
$L = a_0 x + a_1 y$ and $f=b_0 x^3 + 3b_1 x^2 y + 3b_2 xy^2 + b_3 y^3$.
As in Section~\ref{classical}, these bases are chosen so that both $L$ and $f$ are
invariant.  In the column labelled ``LM'' we give the lead monomial of the covariant and in the
final column we give the lead monomial of the corresponding source.

Examining these lead terms we easily see that no one of these 13 covariants can be written as a polynomial in the other 12.
 Thus these 13 covariants minimally generate
$\C[R_1 \oplus R_1 \oplus R_3]^{\SL_2(\C)}$.  Applying Roberts' isomorphism and reducing modulo $p$ yields 13 integral invariants
in $\field[V_2 \oplus V_4]^{C_p}$.  Here $\{x_1,y_1\}$ is a basis of $V_2^*$ and $\{x_2,y_2,z_2,w_2\}$ is a basis
of $V_4^*$.   We use the graded reverse lexicographic ordering with
$w_2 > z_2 > y_2 >  y_1 > x_2 > x_1$.
The lead terms of these 13 $C_p$-invariants are given in the final column of Table~10.1.
We have integral invariants with lead terms $x_1, x_2$ and $y_2^2$.  The lead monomial of
$\tr(w_2^{p-1})$ is $z_2^{p-1}$.  Thus
$\field[V_2\oplus V_4]^{C_p}$ is generated by 13 integral invariants, the two norms $\norm(y_1),
\norm(w_2)$, and the family of transfers: $\tr(w_2^{d_2} z_2^{c_2} y_2^{b_2} y_1^{b_1})$ with
$0 \leq d_2 \leq p-1$, $0 \leq c_2 \leq p-2$, $0 \leq b_2 \leq 1$, $0 \leq b_1 \leq p-1$.

\subsection{The Invariant Ring $\field[V_3 \oplus V_4]^{C_p}$}
  Here we complete the computations discussed in Examples~\ref{eg1} and \ref{eg2} by finding generators for
$\field[V_3 \oplus V_4]^{C_p}$.
   Let $\phi=a_0 x^2 + 2 a_1 xy + a_2 y^2$
 and  $f=b_0 x^3 + 3b_1 x^2 y + 3b_2 xy^2 + b_3 y^3$ denote the quadratic and cubic forms respectively.
 Here we are using $\{x,y\}$ as a basis for $R_1^*$,
 $\{a_0,2a_1,a_2\}$ as a basis for  $R_2^*$ and
$\{b_0,3b_1,3b_2,b_3\}$ as the basis for $R_3^*$.
As in Section~\ref{classical}, these bases are chosen so that both $\phi$ and $f$ are
invariant.  In the column labelled ``LM'' we give the lead monomial of the covariant.

 Generators for the ring of covariants $\C[R_1 \oplus R_2 \oplus R_3]^{\SL_2(\C)}$ are given in \cite[\S 140]{G-Y}.
 There are 15 generators as follows.
 \begin{center}
\begin{tabular}{lcclll}
   Covariant & Order & Bi-degree& LM & LM(Source)\\ \hline
  $\phi$ & 2 & (1,0)& $a_0x^2$& $x_1$\\
  $f $& 3 & (0,1) & $b_0x^3$&$x_2$\\
  $H := (f,f)^2$ & 2 & (0,2) &$b_1^2 x^2$ &$y_2^2$\\
  $T := (f,H)^1$  & 3 & (0,3) & $b_1^3 x^3$ &$y_2^3$\\
  $\Delta := (H,H)^2$  & 0 & (0,4) & $b_1^2 b_2^2$&$y_2^2 z_2^2$\\
  $D:=(\phi,\phi)^2$  & 0 & (2,0) & $a_1^2$&$y_1^2$\\
    $(\phi,f)$  & 3 & (1,1) & $a_1 b_0 x^3$&$x_1 y_2$\\
  $(\phi,f)^2$  & 1 & (1,1) & $a_2 b_0 x$&$x_1 z_2$\\
   $(\phi^2,f)^3$  & 1 & (2,1) & $a_1 a_2 b_0 x$&$ x_1^2 w_2 $\\
  $(\phi^3,f^2)^6$  & 0 & (3,2) & $a_2^3 b_0^2$&$x_1^3 w_2^2$\\
  $(\phi,H)$  & 2 & (1,2) & $a_1 b_1^2 x^2$&$x_1 y_2 z_2$\\
  $(\phi,H)^2$  & 0 & (1,2) & $a_2 b_1^2$&$ x_1 z_2^2 $\\
  $(\phi,T)^2$  & 1 & (1,3) & $a_2 b_1^3 x$&$ x_1 y_2 z_2^2 $\\
  $(\phi^2,T)^3$  & 1 & (2,3) & $a_1 a_2 b_1^3 x$&$ x_1^2 z_2^3 $\\
  $(\phi^3,fT)^6$  & 0 & (3,4) & $a_2^3 b_0 b_1^3$&$ x_1^3 z_2^3 w_2 $\\
   \\
   \multicolumn{5}{c}{Table 10.2. {\bf Covariants of }{\boldmath $R_2 \oplus R_3$}}\\
\end{tabular}
\end{center}
 \newpage 

Examining their lead terms we see that these 15 covariants minimally generate
$\C[R_1 \oplus R_2 \oplus R_3]^{\SL_2(\C)}$.  Applying Roberts' isomorphism and reducing modulo $p$ yields
15 integral invariants in $\field[V_3 \oplus V_4]^{C_p}$.  Here $\{x_1,y_1,z_1\}$ is a basis of $V_3^*$ and
$\{x_2,y_2,z_2,w_2\}$ is a basis of $V_4^*$.  We use the graded reverse lexicographic order with
$w_2 > z_2 > z_1 > y_2 > y_1 > x_2 > x_1$.
The lead terms of these 15 $C_p$-invariants are given in the final column of Table~10.2.
We have integral invariants with lead terms $x_1,x_2,y_1^2$ and $y_2^2$.
Since $\LM(\tr(w_2^{p-1}))=z_2^{p-1}$ we see by Corollary~\ref{explicit key} that
$\field[V_3 \oplus V_4]^{C_p}$ is generated by the 15 integral invariants, the two norms $\norm(z_1),
\norm(w_2)$ and the family of transfers
$\tr( w_2^{d_2} z_2^{c_2} y_2^{b_2} z_1^{c_1} y_1^{b_1})$ with
$0 \leq d_2 \leq p-1$, $0 \leq c_2 \leq p-2$, $0 \leq b_2 \leq  1$, $0 \leq c_1 \leq p-1$ and
$0 \leq b_1 \leq 2$.

 \subsection{Vector Invariants of $V_2$}
    Here we take an arbitrary positive integer $m$ and find generators for $\field_p[m\,V_2]^{C_p}$.
    Suppose the dual of the $i^\text{th}$ copy of $V_2$ is spanned by $\{x_i,y_i\}$ where
    $\Delta(y_i)=x_i$ and $\Delta(x_i)=0$.
    As discussed in the introduction, this ring of invariants was first computed by \name{Campbell} and \name{Hughes}(\cite{campbell-hughes}).
      Recently \name{Campbell}, \name{Shank} and \name{Wehlau}(\cite{CSW}) gave a simplified proof.  Here we give a shorter proof.
      Importantly, the proof in \cite{CSW} yields the stronger and computationally very useful result that the minimal generating set
      for $\field_p[m\,V_2]^{C_p}$ is also a SAGBI basis with respect to a certain term order. 
      
    The integral invariants $\field_p[m\,V_2]^{C_p}$ lift via the Roberts' isomorphism to invariants of $\C[(m+1)\,R_1]^{\SL_2(\C)}$.
    By the first fundamental theorem for $\SL_2(\C)$ (see \cite[\S11.1.2 Theorem~1]{Pr}  for example), this ring is generated by $\binomial{m+1}{2}$
    quadradic determinants $U_{i,j}$ with $0 \leq i < j \leq m$.  Applying Roberts' isomorphism (and reducing
    modulo $p$) yields the integral invariants $u_{0j}=x_j$ for $j=1,2,\dots,m$ and $u_{i,j} = x_i y_j - x_j y_i$
    for $1 \leq i < j \leq m$.
     Thus applying Corollary~\ref{explicit key} we see that $\field_p[m\,V_2]^{C_p}$ is generated by
    \begin{itemize}
           \item$x_j$  for  $j=1,2,\dots,m$;
           \item $\norm(y_i) = y_i^p - x_i^{p-1}y_i$  for $i=1,2,\dots,m$;
           \item $u_{i,j} = x_i y_j - x_j y_i$ for $1 \leq i < j \leq m$;
           \item $\tr(y_1^{a_1} y_2^{a_2} \cdots y_m^{a_m})$  where $0 \leq a_1,a_2,\dots,a_m < p$.
    \end{itemize}

   This set is not a minimal generating set.  \name{Shank} and \name{Wehlau} \cite{cmipg}
   showed that it becomes a minimal generating set if all the transfers
    of degree less than $2p-1$ are omitted.

 \subsection{Vector Invariants of $V_3$}
    Here we take an arbitrary positive integer $m$ and find generators for $\field_p[m\,V_3]^{C_p}$.
    This is the first computation of this ring of invariants.  
    It is possible to adapt the technique used in \cite{CSW} to give a SAGBI basis for $\field_p[m\,V_3]^{C_p}$ (cf.~ \cite{W}).    
    
    Suppose the dual of the $i^\text{th}$ copy of $V_3$ is spanned by $\{x_i,y_i,z_i\}$ where
    $\Delta(z_i)=y_i$, $\Delta(y_i)=x_i$ and $\Delta(x_i)=0$.

 \newpage
    The integral invariants here lift via the Roberts' isomorphism to invariants of
    $\C[R_1 \oplus m\,R_2]^{\SL_2(\C)}$, i.e., to the covariants of $m\,R_2$.\\
      Generators for this ring were found classically.
       See for example \cite[\S 139A]{G-Y}.
       This ring is generated the $\binomial{m+1}{2}$
    quadradic determinants $U_{i,j}=(\phi_i,\phi_j)^1$ with $0 \leq i < j \leq m$ together with
    $\binomial{m+1}{2}$ further
    quadratic polynomials $D_{i,j}=(\phi_i,\phi_j)^2$ with $1 \leq i \leq j \leq m$
    and $\binomial{m}{3}$ determinant invariants $\Det_{i,j,k}$ with $1 \leq i < j < k \leq m$.
    Apply Roberts' isomorphism we get
    \begin{align*}
      \psi(U_{i,j})&=u_{i,j} = x_i y_j - x_j y_i \text{ if } i\ne 0\\
      \psi(U_{0,j})&=x_j\\
      \psi(D_{i,j}) &= d_{i,j} = 2y_i y_j  - 2 z_i x_j - 2 x_i z_j - x_i y_j - y_i x_j\\
      \psi(\Det_{i,j,k}) &= {\rm det}_{i,j,k} = x_i y_j z_k - x_i z_j y_k  + z_i x_j y_k  - y_i x_j z_k +  y_i z_j x_k - z_i y_j x_k
    \end{align*}
   Since $\LM(d_{i,i})=y_i^2$ we have the following theorem.
   \begin{theorem}
    $\field_p[m\,V_3]^{C_p}$ is generated by
   \begin{itemize}
     \item $\norm(z_i)$ for $i=1,2,\dots,m$;
     \item $x_i$  for $i=1,2,\dots,m$;
     \item $u_{i,j}$ with $1 \leq i < j \leq m$;
     \item $d_{i,j}$ with $1 \leq i \leq j \leq m$;
     \item ${\rm det}_{i,j,k}$ with $1 \leq i < j < j \leq m$;
     \item $\tr(\prod_{i=1}^m y_i^{b_i} z_i^{c_i})$ with $0 \leq b_i \leq 1$ and $0 \leq c_i \leq p-1$.
   \end{itemize}
\end{theorem}

 \subsection{Vector Invariants of $V_4$}
    Here we give generators for $\field_p[m\,V_4]^{C_p}$.
    This is the first computation of this ring of invariants.
    Suppose the dual of the $i^\text{th}$ copy of $V_4$ is spanned by $\{x_i,y_i,z_i,w_i\}$ where
     $\Delta(w_i)=z_i$, $\Delta(z_i)=y_i$, $\Delta(y_i)=x_i$ and $\Delta(x_i)=0$.

     Here we need to know generators for $\C[R_1 \oplus m\,R_3]^{\SL_2(\C)}$
 the covariants of $m\,R_3$.    The answer for $m=2$, taken from von Gall \cite{VG1},  is given in Table~10.3. 

  F.~von Gall \cite{VG2} found generating covariants for $3R_3$.
However by results of \name{Schwarz} \cite[(1.22), (1.23) ]{Sch}(see also \cite{Weyl})
we may obtain all the generators of $\C[R_1 \oplus m\, R_3]^{\SL_2(\C)}$ from the
generators of $\C[R_1 \oplus 2\, R_3]^{\SL_2(\C)}$ by the classical process of polarization.
For a description of polarization see for example \cite[Chapter 3 \S 2 ]{Pr} or \cite[page 5]{Weyl}.
Grace and Young \cite[\S 257]{G-Y} also describe another procedure for finding generators for
the covariants of $m\, R_3$.

 \begin{center}
   \begin{tabular}{lcclll}
   Covariant & Order & Bi-degree& LM & LM(Source)\\ \hline
$f_1$&3&(1,0)&$a_0x^3$&$ x_1$\\
$f_2$&3&(0,1)&$b_0x^3$&$x_2$\\
$(f_1,f_2)^3$&0&(1,1)&$a_3b_0$&$x_1w_2$\\  
$H_{20}$&2&(2,0)&$a_1^2x^2$ &$y_1^2$\\
$H_{11}$&2&(0,2)&$a_2b_0x^2$ &$x_1z_2 $\\
$H_{02}$&2&(1,1)&$b_1^2x^2$ &$y_2^2$\\
$U_{12}:=(f_1,f_2)^1$&4&(1,1)&$a_1b_0x^4$ &$x_1y_2$\\   
$(f_1,H_{20})^1$&3&(3,0)&$a_1^3x^3$ &$y_1^3$\\   
$(f_2,H_{02})^1$&3&(0,3)&$b_1^3x^3$ &$y_2^3$\\    
$P:=(f_2,H_{20})^2$&1&(2,1)&$a_2^2b_0x$ &$y_1^2z_2$\\  
$\pi:=(f_1,H_{02})^2$&1&(1,2)&$a_2b_1^2x$ &$x_1z_2^2$\\   
$(f_1,H_{02})^1$&3&(1,2)&$a_1b_1^2x^3$ &$x_1z_2y_2$\\  
$(f_2,H_{20})^1$&3&(2,1)&$a_1a_2b_0x^3$ &$y_1^2y_2$\\  
$(H_{20},H_{20})^2$&0&(4,0)&$a_1^2a_2^2$ &$z_1^2y_1^2$\\  
$(H_{02},H_{02})^2$&0&(0,4)&$b_1^2b_2^2$ &$z_2^2y_2^2$\\   
$(H_{20},H_{02})^2$&0&(2,2)&$a_3^2b_0^2$ &$x_1^2w_2^2$\\   
$(H_{20},H_{11})^2$&0&(3,1)&$a_2^3b_0$ &$y_1^3w_2$\\   
$(H_{02},H_{11})^2$&0&(1,3)&$a_3b_1^3$ &$x_1z_2^3$\\   
$(f_1,P)^1$&2&(3,1)&$a_1a_2^2b_0x^2$ &$y_1^2x_1w_2$\\    
$(f_2,\pi)^1$&2&(1,3)&$a_3b_0b_1^2x^2$ &$x_1z_2^2y_2$\\    
$(H_{20},H_{02})^1$&2&(2,2)&$a_1a_2b_1^2x^2$ &$y_1^2z_2y_2$\\   
$(H_{20},P)^1$&1&(4,1)&$a_1a_2^3b_0x$ &$y_1^4w_2$\\    
$(H_{20},\pi)^1$&1&(3,2)&$a_1a_2^2b_1^2x$ &$y_1^3z_2^2$\\  
$(H_{02},P)^1$&1&(2,3)&$a_2^2b_1^3x$ &$y_1^2z_2^2y_2$\\   
$(H_{02},\pi)^1$&1&(1,4)&$a_3b_1^4x$ &$x_1z_2^3y_2$\\   
$(P,\pi)^1$&0&(3,3)&$a_2^2a_3b_0b_1^2$ &$y_1^2x_1w_2z_2^2$\\  
   \multicolumn{5}{c}{Table 10.3. {\bf Covariants of }{\boldmath $R_3 \oplus R_3$}}\\
\end{tabular}
\end{center}

  It is straight forward to verify that polarization commutes with reduction modulo $p$.
This implies that all the integral invariants of $\field_p[m\,V_4]^{C_p}$
are obtained from polarizing the 26 integral invariants in $\field_p[2\,V_4]^{C_p}$.
In summary, if we let ${w_i,z_i,y_i,x_i}$ denote a basis of the dual of the $i^\text{th}$
copy of $V_4$ where $\Delta(w_i)=z_i$, $\Delta(z_i)=y_i$, $\Delta(y_i)=x_i$, $\Delta(x_i)=0$ we have the following.
\newpage
\begin{theorem}
$\field_p[m\,V_4]^{C_p}$ is generated by
   \begin{itemize}
     \item $\norm(w_i)$ for $i=1,2,\dots,m$;
     \item integral invariants arising from the polarizations of the 26 sources of the $SL_2(\C)$ invariants listed in Table~10.3.
        \item $\tr(\prod_{i=1}^m y_i^{b_i} z_i^{c_i} w_i^{d_i})$ with $0 \leq b_i \leq 1$, $0 \leq c_i \leq p-1$ and $0 \leq d_i \leq p-1$.
   \end{itemize}
\end{theorem}

\begin{remark}
  \name{Shank} showed (\cite[Theorem~3.2]{shank}) that $\LT(\tr(w_i^{p-1}))=z_i^{p-1}$.  Thus we may use $\tr(w_i^{p-1})$ in the 
  role of $f_{i1}$ when we apply Corollary~\ref{explicit key} and hence we have $d_{i1}=p-1$ for all $i=1,2,\dots,m$.  This implies that we may restrict the values of
  the $c_i$ to the range $0 \leq c_i \leq p-2$  in the third family of generators in the above theorem.
\end{remark}

\subsection{Other Representations of $C_p$}
   There are a number of other $SL_2(\C)$ representations for which generators of the ring of covariants are known and thus for which we may
   compute the ring of invariants for the corresponding representation of $C_p$ .
   Here we list some of these representations.

   In 1869, Gordan \cite{G} computed generators for the covariants of the quintic, $R_5$ and the sextic $R_6$.  Grace
   and Young \cite[\S116, \S 134]{G-Y} list these generators.
   In the 1880's F.~von Gall gave generators for the covariants of the septic $R_7$ \cite{vG7}
   and the octic $R_8$ \cite{vG8}.
   Recently  L.~Bedratyuk computed generators for the covariants of the octic \cite{B8} and minimal
   generators for the  covariants of the septic \cite{B7}.   Thus we may list generators for the invariants of
   $V_6$,  $V_7$, $V_8$ and $V_9$.  Although Sylvestor \cite{Syl} published a putative list of generators for the
   covariants of the nonic $R_9$, a recent computation of the invariants of the nonic by
  A. Brouwer and M. Popoviciu,   \cite{B-P} has shown Sylvester's table to be incorrect.  The same two authors 
  have also shown \cite{B-P10} that the ring of invariants of the decimic is generated by 106 invariants which they have 
  constructed.
   Grace and Young \cite[\S 138, \S138A]{G-Y} give a method for obtaining generating covariants for $W\oplus R_1$ and
   $W \oplus R_2$ from the generating covariants of any representation $W$.

\bigskip
\footnotesize
\noindent\textit{Acknowledgments.}
 I thank R.J. Shank, Mike Roth and Gerry Schwarz for many helpful discussions.
  I also thank Megan Wehlau for a number of useful late night conversations which were the genesis of this work.
 This research is supported by grants from ARP and NSERC.

\end{document}